\documentclass[12pt]{amsart}
\usepackage{amsmath}
\usepackage{amssymb}
\usepackage{graphicx}
\usepackage{tikz}
\usepackage{mathtools}
\usepackage{appendix}
\usepackage{hyperref}
\usepackage{multirow}
\usepackage{caption} 
\usetikzlibrary{arrows.meta}

\newtheorem{theorem}{Theorem}[section]
\newtheorem{lemma}[theorem]{Lemma}
\newtheorem{proposition}[theorem]{Proposition}
\newtheorem{corollary}[theorem]{Corollary}

\theoremstyle{definition}
\newtheorem{definition}[theorem]{Definition}
\newtheorem{example}[theorem]{Example}
\newtheorem{remark}[theorem]{Remark}

\DeclareMathOperator{\diam}{diam}

\hypersetup{
    colorlinks=true,
    allcolors=blue,
    pdftitle={Overleaf Example},
    pdfpagemode=FullScreen,
}
 
\title{Multiplicative Invariance for a Class of Subsets of the Complex Plane}
\author{Neil MacVicar}

\begin{document} %% MAIN DOCUMENT
  
\maketitle

\begin{abstract} %%ABSTRACT
Multiplicative invariance is a well-studied property of subsets of the unit interval. The theory in the complex plane is less developed. This paper introduces an analogous definition for multiplicative invariance in the complex plane coinciding with a more general definition concerning subsets of attractors of iterated function systems satisfying the strong separation condition. We establish similar results to those of Furstenberg's in the unit interval. Namely, that the Hausdorff and box-counting dimensions of a multiplicatively invariant set are equal and, furthermore, are equal to the normalized topological entropy of an underlying subshift. We also extend results concerning the box-counting dimension of intersections of base-$b$ restricted digit sets with their translates where $b$ is a suitably chosen Gaussian integer. 
\end{abstract}

%% THE GOAL OF THE PAPER (SECTION 1-4 AT LEAST), 5 CAN BE THE EXTENSION

\section*{\textbf{Introduction}} %%INTRODUCTION

Throughout his career, Furstenberg made contributions to many areas of mathematics using dynamical methods. Among those contributions is a pair of papers at the intersection of dynamics and fractal geometry (\cite{F67}, \cite{F70}). Therein, Furstenberg proved results and made conjectures about the fractal properties of multiplicatively invariant subsets of the unit interval. Multiplicatively invariant subsets are those that are invariant under the map $x \mapsto rx \mod 1$ where $r$ is some positive integer. For a specific value $r$, this is called $\times r$-invariance. The following theorem highlights particular results of Furstenberg which are recalled in Section~\ref{section:realCase} of this paper. 
\begin{theorem}\label{thm:foundation}[H. Furstenberg, \cite{F67}, proposition III.1]%%WHAT I REPLICATE
Let $r\geq2$ be an integer. Let $\mathcal{E}$ denote topological entropy, let $\dim_{H}$ denote Hausdorff dimension, and let $\dim_{B}$ denote box-counting dimension. If $A \subset \{0,1,\ldots,r-1\}^{\mathbb{N}}$ is a subshift, then 
\begin{itemize}
\item[(i)]$\tilde{A} = \{\sum_{k=1}^{\infty}a_{k}r^{-k}:(a_{k})_{k\geq1}\in A\}$ is $\times r$-invariant, and
\item[(ii)]$\dim_{B}\tilde{A} = \frac{\mathcal{E}(A)}{\log{r}}.$
\item[(iii)] If $Y$ is a $\times r$-invariant set, then $\dim_{H}Y = \dim_{B}Y.$
\end{itemize}
\end{theorem}

Considerable development of the theory of multiplicatively invariant subsets of the unit interval has been pursued since: Furstenberg's sumset conjecture, which offers sufficient conditions under which the Hausdorff and box-counting dimensions of sumsets of multiplicative invariant subsets split into the sum of the dimensions of those subsets, was proven by Hochman and Shmerkin in \cite{HS12}. Additionally, Furstenberg's intersection conjecture (now known as the Shmerkin-Wu theorem) was proven independently by Shmerkin in \cite{S19} and Wu in \cite{W19} using different methods and again by Austin in \cite{A22}. 

In \cite{GMR21}, Richter, Moreira, and Glasscock established similar results to those of Furstenberg in \cite{F67} and a sumset result for a version of $\times r$-invariance for subsets of the nonnegative integers. 

The theory in the complex plane is less developed. In \cite{PS21}, Pedersen and Shaw study a complex analogue of a class of multiplicative invariant subsets called base-$r$ restricted digit Cantor sets. A base-$r$ restricted digit Cantor set contains those numbers in the unit interval that, when written in base-$r$, restrict the coefficients used in their expansions to some subset of $\{0, 1,\ldots, r-1\}$. For example, the middle-thirds Cantor set are all numbers in the unit interval that, when written in base $3$, only use the coefficients $0$ and $2$. 

The problem of defining a more general class of sets that might be called ``$\times b$-invariant" where $b$ is a Gaussian integer presents challenges that differ from the real case. The map used to define multiplicative invariance in the unit interval subtracts the integer part to ensure that the image is in the domain. It is not immediately clear what the correct choice is for the integer part of a complex number. Our approach is to first write a Gaussian integer in base $b$. Representing complex numbers with respect to a Gaussian integer base can be traced back to Knuth in \cite{K60} as an alternative way of storing complex numbers and performing complex arithmetic on computers. Canonical number systems with Gaussian integers was explored by Katai and Szabo in \cite{KS74}. More generally, canonical number systems for quadratic fields were both studied by Gilbert in \cite{G81} and Katai and Kov\'{a}cs in \cite{KK81}. The geometry of tiles generated by those systems has also been studied (see, for example, \cite{AT05} by Akiyama and Thuswaldner). 

This paper introduces a definition for invariance concerning subsets of iterated function systems (Definition~\ref{def:contractInvariant}) from which a definition for $\times b$-invariance (Definition~\ref{def:timesb}) is given. Our main result is Theorem \ref{thm:earlymain} (Theorem \ref{thm:main} in Section \ref{section:mainResult}). It is similar to Theorem~\ref{thm:foundation}. Here $C_{D}$ denotes the set of $z = d_{1}b^{-1} + d_{2}b^{-2} +\cdots$ where the coefficients $d_{k}$ are elements of $D\subset\{0, 1,\ldots, |b|^{2}-1\}$ and $b=-n+i$ for some positive integer $n$. 
\begin{theorem}\label{thm:earlymain} %%WE MIGHT AS WELL BE THOROUGH
Let $b=-n+i$ with $n\geq2$ and assume $D\subset\{0, 1,\ldots, |b|^{2}-1\}$ is nonempty and satisfies $|d-d^{'}|\neq 1$ for all $d, d^{'}\in D$. If $A\subset D^{\mathbb{N}}$ is a subshift, then 
\begin{itemize}
\item[(i)]$\tilde{A} = \{\sum_{k=1}^{\infty}a_{k}b^{-k}:(a_{k})_{k\geq1}\in A\}$ is $\times b$-invariant, and
\item[(ii)]$\dim_{B}\tilde{A} = \frac{\mathcal{E}(A)}{\log{|b|}}.$
\item[(iii)] If $Y\subset C_{D}$ is a $\times b$-invariant set, then $\dim_{H}Y = \dim_{B}Y.$
\end{itemize}
\end{theorem}

In addition to this, we extend the application of a formula for the box-counting dimension of the intersection of $C_{D}$ with a translate of itself. This formula was originally presented in \cite{PS21}. These kinds of intersections have been studied for subsets of the real line (see \cite{DH95}, \cite{NL02}, or \cite{PP12}).

\begin{theorem}\label{thm:earlyminimalSep} %% PEDERSEN AND SHAW EXTENDED RESULT
Let $b = -n+i$ with $n\geq2$. Let $D\subset\{0, 1, \ldots, |b|^{2}-1\}$ satisfy $d\leq n^{2}/2$ for all $d\in D$ and $|a-a^{'}|\neq 1$ for all $a, a^{'}\in D-D$. We have
\begin{equation}
\underline{\dim}_{B}(C_{D} \cap (C_{D}+z)) = \liminf_{m\to\infty}\frac{\log G_{m}(z)}{m\log|b|}
\end{equation}
where 
\begin{equation}
G_{m}(z) := |D\cap(D+z_{1})||D\cap(D+z_{2})|\cdots|D\cap(D+z_{m})|
\end{equation}
and $z = 0.z_{1}z_{2}\ldots$ with $z_{k}\in D-D$. 
\end{theorem}

The original statement in \cite{PS21} assumes every pair of distinct elements of $D$ is at least distance $n+1$ apart. Our extension of the formula coincides with the extension of one of its corollaries. Let $F = \{z: C_{D}\cap(C_{D}+z)\neq\emptyset\}$. For $0\leq\beta\leq1$, let us consider those elements of $z\in F$ such that $\dim_{B}(C_{D}\cap(C_{D}+z)) = \beta\dim_{B}C_{D}$. 

\begin{corollary} \label{cor:earlymoreDenseTranslates}
Assume the hypotheses of Theorem \ref{thm:earlyminimalSep}. The set $F_{\beta}$ is dense in $F$ for any $\beta\in[0, 1]$. 
\end{corollary}

\section*{\textbf{Organization}} %%WHAT EACH SECTION CONTAINS
This paper is separated into five sections and two appendices. 
\begin{enumerate}
\setlength\itemsep{1em}
\item[(1)] Section~\ref{section:realCase} reviews the basics of multiplicative invariance in the unit interval and includes concepts from fractal geometry and symbolic dynamics that are present throughout the paper. 
\item[(2)] Section~\ref{section:IFSs} introduces a definition for a kind of invariance concerning iterated function systems and develops tools that are used to prove the main result in Section \ref{section:mainResult}.
\item[(3)] Section~\ref{section:complexCase} defines $\times b$-invariance and includes background on base-$(-n+i)$ expansions of complex numbers.
\item[(4)] Section~\ref{section:mainResult} includes the proofs of the statements in the main result (Theorem \ref{thm:main}). 
\item[(5)] Section~\ref{section:intersectionResults} extends results on the box-counting dimension of intersections of restricted digit Cantor sets and their translates (Theorem \ref{thm:minimalSep}). 
\item[(A)] Appendix \ref{app:a} illustrates the derivation of the rules governing base-$(-n+i)$ expansions when $n\geq3$. 
\item[(B)] Appendix \ref{app:b} includes the rules governing the special case of base-$(-2+i)$ expansions. 
\end{enumerate}

\section{\textbf{Multiplicative Invariance in $\mathbb{R}$}} \label{section:realCase} %%SECTION ONE MIGHT NEED TO CHANGE SECTIONS TO CHAPTERS - CONTENT: WILL BE LISTING BARE BONES WITHOUT TRANSITIONS

In this section we recall multiplicative invariance for subsets of the unit interval and review their fractal properties which inspired the main result. 

\begin{definition} \label{def:timesr} %%TIMES R INVARIANCE %%MADE A CHANGE HERE TO T_R (TO RESTRICT IT, WE NEED TO TALK ABOUT THE CIRLE)
Let $r$ be a positive integer. Define the map
\begin{equation}
\begin{split}&T_{r} : \mathbb{R} \rightarrow [0, 1) \\
&x \mapsto rx\mod{1}
\end{split}
\end{equation}
A nonempty closed subset $Y \subset [0, 1]$ is called $\times r-$\textit{invariant} if $T_{r}(Y) \subset Y$. A subset $Y$ is called \textit{multiplicatively invariant} if it is $\times r$-invariant for some $r\geq2$.
\end{definition}

\begin{example} %%RESTRICTED DIGIT CANTOR SET / CANTOR RESTRICTED DIGIT SET 
Let $r$ be a positive integer. Suppose $D$ is a nonempty subset of $\Lambda_{r}:=\{0, 1,\ldots, r-1\}$. We call the set 
\begin{equation}
C_{r, D} := \bigg\{\sum_{k\geq1}d_{k}r^{-k}\in\mathbb{R} : d_{k}\in D\bigg\}\label{equation}
\end{equation}
the \textit{base}-$r$ \textit{restricted digit Cantor set with digit set} $D$. These sets are $\times r$-invariant. 
\end{example}
We bring attention to a slight abuse of terminology. This example includes both the cases when $D$ is either a singleton or equal to $\Lambda_{r}$. The set $C_{r, D}$ is not a Cantor set in these two extreme cases.  

The fractal properties of multiplicatively invariant sets are expressed through their Hausdorff and box-counting dimensions. We recall these dimensions here. 

\begin{definition} %%DELTA COVER
Let $\delta > 0$ and $V$ be a subset of a metric space $X$. A countable collection of sets $\{U_{k} \subset X\}$ is called a $\delta$-\textit{cover} of $V$ if 
\begin{enumerate}
\item[(i)] $V \subset \bigcup_{k}U_{k}$,
\item[(ii)] $\diam{U_{k}} \leq \delta$ for each $k$. 
\end{enumerate}
\end{definition}

\begin{definition} %%HAUSDORFF MEASURE AND HAUSDORFF DIMENSION
Let $V$ be a subset of a metric space and let $s > 0$. For every $\delta > 0$, define the quantity
\begin{equation}
\mathcal{H}_{\delta}^{s}(V) := \inf{\bigg\{\sum_{k}(\diam{U_{k}})^{s} : \{U_{k}\}\: \text{is a $\delta$-cover of}\: V \bigg\}}.
\end{equation}
The $s$-\textit{dimensional Hausdorff measure} of $V$ is the limiting value $\mathcal{H}^{s}(V) := \lim_{\delta\rightarrow 0^{+}}\mathcal{H}_{\delta}^{s}(V)$. 

We call the quantity
\begin{equation}
\dim_{H}{V} := \inf{\{s\geq0 : \mathcal{H}^{s}(V) = 0\}}
\end{equation}
the \textit{Hausdorff dimension} of $V$. 
\end{definition}

The Hausdorff dimension can be equivalently defined using less general covers. For example, it is common to add the condition that the $\delta$-covers only contain balls.
\begin{proposition} \label{rem:balls} [K. Falconer, \cite{F90}, section 2.4] %%CAN RESTRICT DELTA COVERS TO COVERS OF DELTA BALLS
Let $V$ be a subset of a metric space and define 
\begin{equation}
\mathcal{B}_{\delta}^{s}(V) :=  \inf{\bigg\{\sum_{k}(\diam{B_{k}})^{s} : \{B_{k}\}\: \text{is a $\delta$-cover of}\: V\:\text{by balls} \bigg\}}.
\end{equation}
Then $\dim_{H}{V} := \inf{\{s\geq0 : \mathcal{B}^{s}(V) = 0\}}$ where $\mathcal{B}^{s}(V) =  \lim_{\delta\rightarrow 0^{+}}\mathcal{B}_{\delta}^{s}(V)$.
\end{proposition}

The Hausdorff dimension exhibits desirable properties, but it is difficult to compute directly. The box-counting dimension is a popular alternative because of the comparative ease of computing it. 

\begin{definition}\label{def:boxcounting} %%BOX COUNTING DIMENSION
Let $V$ be a subset of a metric space $X$. Let $\delta > 0$. Let $N_{\delta}(V)$ denote the minimum number of subsets of $X$ of diameter at most $\delta$ required to cover $V$. If it exists, we call the limit
\begin{equation}
\dim_{B}V := \lim_{\delta\rightarrow0^{+}}\frac{\log{N_{\delta}(V)}}{-\log{\delta}}
\end{equation}

the\textit{ box-counting dimension} of $V$. 
\end{definition}

In the event the limit does not exist, we refer to the upper and lower limits of the above function of $\delta$ as the \textit{upper} and \textit{lower} box-counting dimensions respectively. This fractal dimension is useful because the $N_{\delta}$ function has several equivalent formulations (see \cite{F90} section 3.1 for a list). In particular, we use the fact that we may replace $N_{\delta}$ by the function that takes a set $V$ to the minimum number of closed balls of radius $\delta$ needed to cover it in the proof of Lemma \ref{lem:boxtile}.

Multiplicatively invariant subsets of the unit interval are also connected to subshifts. We recall the relevant definitions. 

\begin{definition}\label{def:subshifts}%%SUBSHIFTS
Let $\Omega$ be a finite set equipped with the discrete topology. Let $\Sigma = \Omega^{\mathbb{N}}$ be the sequence space equipped with the product topology and define the left shift map 
$$\sigma: \Sigma\rightarrow\Sigma$$ 
$$(\omega_{k})_{k\geq1} \mapsto (\omega_{k+1})_{k\geq1}.$$

We call $A\subset X$ a \textit{subshift} if it is closed and satisfies $\sigma(A)\subset A$. 
\end{definition}

\begin{definition} \label{def:topent} %%TOPOLOGICAL ENTROPY OF A SUBSHIFT
Let $A$ be a subshift. The \textit{topological entropy} of $A$ is the limit 
\begin{equation}
\mathcal{E}(A) :=  \lim_{n\rightarrow\infty}\frac{\log|\mathcal{L}_{n}(A)|}{n}
\end{equation}
where $\mathcal{L}_{n}(A) := \{(a_{1}, a_{2}, \ldots, a_{n}) : a_{1} = \omega_{1},\ldots,a_{n}=\omega_{n}\;\textit{for some}\; (\omega_{k})_{k\geq1}\in A\}$. 
\end{definition}

We remark that a more general definition of topological entropy can be found in chapter 7 section 1 of \cite{W70} for continuous maps defined on compact spaces. This more general formulation is shown in theorem 7.13 of \cite{W70} to reduce to the formula above in the case of subshifts. In particular, the limit exists. 

We now state a result of Furstenberg's (\cite{F67}, proposition III.1) about multiplicatively invariant subsets of $[0, 1]$ in two parts. We state similar results for a class of subsets of $\mathbb{C}$ in section \ref{section:mainResult}. 

\begin{theorem}\label{thm:dimtop} [H. Furstenberg,  \cite{F67}, proposition III.1] %%DIMENSION == TOPOLOGICAL ENTROPY
Let $r\geq2$ be an integer. If $A \subset \Lambda_{r}^{\mathbb{N}}$ is a subshift, then 
\begin{itemize}
\item[(i)]$\pi(A) = \{\sum_{k\geq1}a_{k}r^{-k}:(a_{k})_{k\geq1}\in A\}$ is $\times r$-invariant, 
\item[(ii)]$\dim_{B}\pi(A) = \frac{\mathcal{E}(A)}{\log{r}},$
\end{itemize}
where $\pi:\Lambda_{r}^\mathbb{N} \rightarrow \mathbb{R}$ is given by $(\omega_{k})_{k\geq1} \mapsto \sum_{k\geq1}\omega_{k}r^{-k}$. 
\end{theorem}

\begin{theorem}\label{thm:hb} [H. Furstenberg, \cite{F67}, proposition III.1] %%DIMENSIONAL EQUIVALENCE
Let $Y$ be a $\times r$-invariant set. Then $\dim_{H}Y = \dim_{B}Y.$
\end{theorem}

\begin{remark}%GOING FROM INVARIANT SETS TO SUBSHIFTS
In \cite{F67}, proposition III.1 states that the Hausdorff and box-counting dimensions of the set $\pi(A)$ in Theorem~\ref{thm:dimtop} are equal. The preimage of a $\times r$-invariant set under $\pi$ is a subshift of $\Lambda_{r}^{\mathbb{N}}$ and hence we can claim the equality for Hausdorff and box-counting dimensions for all $\times r$-invariant sets. 
\end{remark}

\begin{example} %%RESTRICTED DIGIT CANTOR SETS / MIDDLE THIRDS
The middle-third Cantor set is the image of the set of sequences $\{(a_{k})_{k\geq1}:a_{k}\in{0, 2}\}$ under the map $(a_{k})_{k\geq1}\mapsto\sum_{k\geq}a_{k}3^{-k}$ in Theorem~\ref{thm:dimtop}. The topological entropy of this subshift according to Definition~\ref{def:topent} is $\log{2}$. It follows from the previous two theorems that $\dim_{H}C_{3,\{0,2\}} = \dim_{B}C_{3,\{0, 2\}} = \log{2}/\log{3}$.  
\end{example}

The proof of Theorem \ref{thm:hb} in \cite{F67} makes use of a technical fact about subshifts. We use this fact in our proof of Theorem \ref{thm:dimEq}. To state it we introduce the following constructions. 

Let $A\subset \Sigma = \Omega^{\mathbb{N}}$ be a subshift. Let $L = \cup_{n\geq1}\Omega^{n}$. This is the set of all finite words written using the alphabet $\Omega$. Let $R$ be the subset of $L$ containing those tuples which occur as (finite) subwords of sequences in $A$. The set $\mathcal{L}_{n}(A)$ can be viewed as the elements of $R$ of length $n$. We denote the length of a word $\rho$ by $l(\rho)$. The set $R$ is a semigroup under concatenation. Let us say that a word $\rho$ \textit{divides} a word $\rho^{'}$ if $\rho^{'} = \rho\rho_{1}$ for some $\rho_{1}\in L$. 

\begin{lemma} \label{lem:semigroup} %% FURSTENBERG'S MAIN TRICK
Let $A\subset \Sigma$ be a subshift. Let $R$ be the collection of all finite subwords of sequences in $A$. Suppose there exists a finite collection of subwords $\{\rho_{k}\}_{k=1}^{K}\subset R$ such that whenever $\rho\in R$ is of sufficient length, it is divisible by $\rho_{k}$ for some $k$. If 
$\sum_{k=1}^{K}r^{sl(\rho_{k})} < 1$ where $s > 0$ and $r\in(0, 1)$, then $\sum_{R} r^{sl(\rho)}$ converges. 
\end{lemma}

\begin{proof}
Let $\langle\rho_{k}\rangle$ be the semigroup generated by $\{\rho_{k}\}_{k=1}^{K}$ using concatenation. We have 
\begin{equation}\label{eq:semifinite}
\begin{split}
\sum_{\langle\rho_{k}\rangle}r^{sl(\rho_{k_{1}}\rho_{k_{2}}\ldots\rho_{k_{n}})} &= \sum_{n=1}^{\infty}\sum_{(k_{1}, k_{2},\ldots, k_{n})}r^{sl(\rho_{k_{1}}\rho_{k_{2}}\ldots\rho_{k_{n}})} \\
&=\sum_{n=1}^{\infty}\bigg(\sum_{(k_{1}, k_{2},\ldots, k_{n})}\prod_{i=1}^{n}r^{sl(\rho_{k_{i}})}\bigg) \\
&=\sum_{n=1}^{\infty}\bigg(\sum_{k=1}^{K}r^{sl(\rho_{k})}\bigg)^{n}. \\
\end{split}
\end{equation}
The last sum is a convergent geometric series by our assumption that $\sum_{k=1}^{K}r^{sl(\rho_{k})} < 1$. We can use this to prove that $\sum_{R}r^{sl(\rho)}$ converges. It is always the case that if $\rho = \rho_{1}\rho_{2}\in R$, then $\rho_{1}\in R$. By the shift invariance of $A$, it must also be that $\rho_{2}\in R$. By assumption, the set $\{\rho_{k}\}_{k=1}^{K}$ has the property that every element of $R$ of length greater than some $N$ is divisible by one of the elements of $\{\rho_{k}\}_{k=1}^{K}$. Combining these two properties allows us to divide until there is no more room to do so. This yields
\begin{equation}
\rho = \rho_{k_{1}}\rho_{k_{2}}\ldots\rho_{k_{n}}\rho_{j}^{'}
\end{equation}
where $\rho_{j}^{'}$ is some element of $R$ that is of insufficient length to be divided further. The set of these remainders is finite since there are only finitely many words whose length is less than $N$, say $J$ of them. It suffices to argue that $\sum_{R}r^{sl(\rho)}$ is finite when we restrict the index set to those words of length at least $N$. Observe that
\begin{align}
\sum_{\rho\in R, l(\rho)\geq N}r^{sl(\rho)} &= \sum_{\rho\in R, l(\rho)\geq N}r^{sl(\rho_{k_{1}}\rho_{k_{2}}\ldots \rho_{k_{n}}\rho_{j}^{'})} \\
&< \sum_{\langle\rho_{k}\rangle}\sum_{j=1}^{J}r^{sl(\rho_{k_{1}}\rho_{k_{2}}\ldots\rho_{k_{n}}\rho_{j}^{'})} \\
&< J\sum_{\langle\rho_{k}\rangle}r^{sl(\rho_{k_{1}}\rho_{k_{2}}\ldots\rho_{k_{n}})}.\label{eq:boundingSeries}
\end{align}
The last quantity is finite since the quantity in~(\ref{eq:semifinite}) is finite. This implies that $\sum_{R}r^{sl(\rho)}$ converges.
\end{proof}

\section{\textbf{Invariance for Iterated Function Systems}} \label{section:IFSs} %% ITERATED FUNCTIONS SYSTEMS

This section contains a number of definitions and lemmas concerning iterated function systems that are used to prove our main result (Theorem~\ref{thm:main}). 

%% DEFINE AN IFS => Hutchinson's theorem etc...

Let $(X, d)$ be a metric space. A map $f:X\to X$ is called a \textit{contraction} if there exists $c\in(0, 1)$ such that for all $x, y\in X$ we have $d(f(x), f(y)) \leq cd(x, y)$.

A finite collection of contractions $\mathcal{F}$ is called an iterated function system (IFS). We can use $\mathcal{F}$ to define a map $F:\mathcal{P}(X)\to\mathcal{P}(X)$ by setting
\begin{equation}
F(S) = \cup_{f\in\mathcal{F}}f(S),
\end{equation}
for any $S\subset X$. The map $F$ is sometimes referred to as the Hutchinson operator. 

Hutchinson proved in \cite{H81} that there exists, among the class of nonempty compact subsets of a complete metric space, a unique set invariant under $F$. 
\begin{theorem} \label{thm:hutchinson} [J.E. Hutchinson, \cite{H81}, theorem 1]%%HUTCHINSON'S THEOREM
Let $X$ be a complete metric space and $\mathcal{F}$ be an IFS defined on $X$. There exists a unique nonempty compact set $E\subset X$ such that 
$F(E) = E$.  
\end{theorem}

The set $E$ is called the \textit{attractor} of the IFS. There is a natural way of relating the attractor to a symbolic space. This is accomplished by viewing the IFS $\mathcal{F}$ as a finite alphabet. 

\begin{definition} %% CODING SPACE
Let $\mathcal{F} = \{f_{1}, f_{2}, \ldots, f_{n}\}$ be an IFS defined on a complete metric space $X$. Let $E$ denote the attractor of $\mathcal{F}$ and $\Sigma = \{1, 2,\ldots, n\}^{\mathbb{N}}$. The map $\pi:\Sigma\to E$ defined by
\begin{equation}
\pi((a_{k})_{k\geq1}) = \cap_{m\geq1}(f_{a_{1}}\circ f_{a_{2}}\circ \cdots \circ f_{a_{m}})(E)
\end{equation}
is called the \textit{coding map} and $\Sigma$ is referred to as the \textit{coding space}.
\end{definition}

We can see that the coding map is surjective by considering the orbit of the attractor under $F$. 

If the attractor $E$ is a subset of euclidean space then the coding map is equivalent to $(a_{k})_{k\geq1} \mapsto \lim_{m\to\infty}(f_{a_{1}}\circ f_{a_{2}}\circ \cdots \circ f_{a_{m}})(0)$. For a fixed word $a_{1}a_{2}\cdots a_{m}$ we call the cylinder set $E_{a_{1}, a_{2},\ldots, a_{m}} := (f_{a_{1}}\circ f_{a_{2}}\circ \cdots \circ f_{a_{m}})(E)$ an \textit{$m$-tile}. Using this language, the coding map takes a decreasing sequence of $m$-tiles and maps them to the unique point in their intersection. %% UNIQUENESS VIA CANTOR'S INTERSECTION THEOREM

There are a number of special classes of IFS. In this document we focus on those that satisfy the strong separation condition.

\begin{definition} %% STRONG SEPARATION
Let $\mathcal{F}$ be an IFS defined on a complete metric space $X$. Let $E$ denote the attractor of $\mathcal{F}$. The IFS $\mathcal{F}$ satisfies the \textit{strong separation condition} if for every pair of distinct maps $f_{1}, f_{2}\in\mathcal{F}$, we have $f_{1}(E)\cap f_{2}(E) = \emptyset$. 
\end{definition}

The coding map $\pi$ is injective under the strong separation condition. This can be used to define a map $T: E \to E$ given by $T = \pi \circ \sigma \circ \pi^{-1}$. Here, $\sigma$ denotes the left shift on the coding space. 

\begin{definition} \label{def:contractInvariant}%% INVARIANCE FOR IFS WITH THE STRONG SEPARATION CONDITION
Let $X$ be a complete metric space. Suppose that $E \subset X$ is the attractor of an IFS satisfying the strong separation condition. A nonempty closed subset $K$ of $E$ is called \textit{$\times (c_{1}, c_{2}, \ldots, c_{n})$-invariant} if $T(K) \subset K$. Here the numbers $c_{1}, c_{2}, \ldots, c_{n}$ are the contraction coefficients associated with the IFS. If the IFS is homogeneous ($c_{1} = c_{2} = \cdots = c_{n} = c$), then we simply call the set $K$ \textit{$\times c$-invariant}. 
\end{definition}

\begin{lemma} \label{lem:entropyIFS} %% "ENTROPY OF AN IFS"
Let $X$ be a complete metric space. Suppose that $E\subset X$ is the attractor corresponding to an IFS that satisfies the strong separation condition. Let $\Sigma$ be the coding space associated with this system and let $\pi$ be the coding map. Suppose that $A\subset\Sigma$ is a subshift. 
\begin{itemize}
\item[(i)] The set $\pi(A)$ is $\times (c_{1}, c_{2}, \ldots, c_{n})$-invariant, where $c_{1}, c_{2}, \ldots, c_{n}$ are the contraction coefficients associated with the IFS. 
\item[(ii)] We have $\mathcal{E}(A)= \lim_{m\to\infty}\frac{\log{N_{m}(\pi(A))}}{m}$, where $N_{m}(\pi(A))$ denotes the smallest number of $m$-tiles required to cover $\pi(A)$ and $\mathcal{E}(A)$ is the topological entropy of $A$. 
\end{itemize}

Moreover, if $K\subset E$ is $\times (c_{1}, c_{2}, \ldots, c_{n})$-invariant, then $\pi^{-1}(K) \subset \Sigma$ is a subshift. 
\end{lemma}

\begin{proof} 
Since the coding map is continuous and $A$ is compact, we obtain that $\pi(A)$ is a compact subset of $X$ and therefore is closed. To see the invariance, let $x\in \pi(A)$. We have $x = \pi((a_{k})_{k\geq1})$. If $T = \pi\circ\sigma\circ\pi^{-1}$ where $\sigma$ is the left shift operator on the coding space, then we see that
$T(x) = \pi((a_{k+1})_{k\geq1})$. The sequence $(a_{k+1})_{k\geq1}$ is an element of $A$ since $A$ is a subshift and so we see that $\pi(A)$ is invariant under $T$. This proves claim $(i)$. This same entwining of $\pi$ and $\sigma$, in addition to the continuity of $\pi$, shows that the preimage of a $\times (c_{1}, c_{2}, \ldots, c_{n})$-invariant subset $K$ under $\pi$ is a subshift. 

We now observe that the topological entropy of $A$ can be expressed using covers of $\pi(A)$ by $m$-tiles. The assumption that $\mathcal{F}$ satisfies the strong separation condition implies that $\pi$ is a bijection. The coding map then induces a bijective correspondence between the cylinder sets $[a_{1},a_{2},\ldots,a_{m}]$ and the $m$-tiles. This is a bijection between $m$-tiles and the subwords of length $m$ in $A$. Using the notation developed for subshifts in Section \ref{section:realCase}, we have 
\begin{equation}
N_{m}(\pi(A))=|\mathcal{L}_{m}(A)|
\end{equation}
and in particular,
\begin{equation}
\frac{\log{N_{m}(\pi(A))}}{m} = \frac{\log{|\mathcal{L}_{m}(A)|}}{m}.
\end{equation}
Taking the limit as $m\rightarrow\infty$ yields the result. 
\end{proof} 

\begin{definition} \label{def:tileMeasure} %% HAUSDORFF MEASURE WITH TILES
Let $X$ be a complete metric space. Suppose that $E\subset X$ is the attractor corresponding to an IFS $\mathcal{F}$.
For $s, \delta >0$ and $V\subset E$, we define the quantity
\begin{equation}
\mathcal{T}_{\delta}^{s}(V) := \inf{\bigg\{\sum_{k=1}^{\infty}(\diam{T_{k}})^s : \{T_{k}\}\; \text{is a $\delta$-cover of}\;V\;\text{where each $T_{k}$ is an $m_{k}$-tile}\bigg\}}.
\end{equation}
We denote the limit $\lim_{\delta\rightarrow0^{+}}\mathcal{T}_{\delta}^{s}(V)$ by $\mathcal{T}^{s}(V)$.
\end{definition}

\begin{lemma} \label{lem:francescowasright} %%HAUSDORFF DIMENSION WITH M-TILE
Let $X$ be a complete metric space and $E\subset X$ be the attractor of a homogeneous IFS with contraction coefficient $c$. Suppose that the number of $m$-tiles that a ball of diameter less than or equal to $c^{m}\diam{E}$ is bounded by a constant independent of $m$. For any $V\subset E$, we have $\dim_{H}{V} = \inf{\{s\geq0 : \mathcal{T}^{s}(V) = 0\}}$.
\end{lemma}
\begin{proof}
Suppose that $\{B_{k}\}$ is a $\delta$-cover of $V$ by balls. Since our ultimate concern is with the limit as $\delta$ tends to zero, we assume $\delta\in(0, 1)$. 

For each $k$, we can find an integer $m_{k}$ such that $c^{m_{k}+1}\diam{E} < \diam{B_{k}}\leq c^{m_{k}}\diam{E}$. The collection of these $m_{k}$-tiles, over all $k$, form a $c^{-1}\delta$-cover of $V$. Let $T^{(k)}_{j}$ denote the $j$th $m_{k}$-tile that intersects $B_{k}$. Let $M$ be the upper bound on the number of $m_{k}$-tiles that $B_{k}$ can intersect. For $s>0$ we have
\begin{align}
\sum_{k}\sum_{j}(\diam{T^{(k)}_{j}})^{s} &\leq \sum_{k}M(\diam{T^{(k)}_{1}})^{s} \\
&=  M\sum_{k}(c^{m_{k}}\diam{E})^{s} \\
&= Mc^{-s}\sum_{k}(c^{m_{k}+1}\diam{E})^{s} \\
&\leq Mc^{-s}\sum_{k}(\diam{B_{k}})^{s}.
\end{align}
Since $\{T^{(k)}_{j}\}$ is a collection of $m_{k}$-tiles that form a $c^{-1}\delta$-cover of $V$, we obtain
\begin{equation}
\mathcal{T}_{c^{-1}\delta}^{s}(V) \leq Mc^{-s}\sum_{k}(\diam{B_{k}})^{s}.
\end{equation}
Since the $\delta$-cover of balls is arbitrary, this implies $\mathcal{T}_{c^{-1}\delta}^{s}(V) \leq Mc^{-s}\mathcal{B}_{\delta}^{s}(V)$ (see Proposition~\ref{rem:balls} to recall this notation). The Hausdorff measure is defined using arbitrary countable $\delta$-covers and so we immediately have $\mathcal{H}_{c^{-1}\delta}^{s}(V)\leq \mathcal{T}_{c^{-1}\delta}^{s}(V)$. Taking limits as $\delta\rightarrow0^{+}$ yields
\begin{equation}
\mathcal{H}^{s}(V)\leq \mathcal{T}^{s}(V) \leq Mc^{-s}\mathcal{B}^{s}(V).
\end{equation}
Both $\mathcal{H}^{s}(V)$ and $\mathcal{B}^{s}(V)$ have the property that they are $+\infty$ for $s < \dim_{H}V$ and $0$ for $s > \dim_{H}V$. It follows that $\mathcal{T}^{s}(V)$ shares this property. Therefore 
\begin{equation}
\inf{\{s\geq0:\mathcal{T}^{s}(V) = 0\}} = \inf{\{s\geq0:\mathcal{H}^{s}(V) = 0\}} = \dim_{H}V.
\end{equation}
\end{proof}

%% WILL CALL THIS SECTION 3
%% I THINK IT MAKES SENSE TO DEVELOP IFS TOOLS IN THE PREVIOUS SECTION THEN MOVE TO THE PARTICULAR CONTEXT OF APPLICATION

\section{\textbf{Multiplicative Invariance in $\mathbb{C}$}} \label{section:complexCase} %%MY WORK IN C

In this section, we define $\times b$-invariance for a class of subsets of the complex plane where $b$ is some Gaussian integer. This will be analogous to Definition~\ref{def:timesr} ($\times r$-invariance). The classical examples of $\times r$-invariant sets are the restricted digit Cantor sets. Those sets are captured by restricting digits in a specified number system. For example, the middle-thirds Cantor set is the set of numbers in the unit interval whose ternary expansions do not use the digit $1$. We proceed similarly by presenting a number system for writing complex numbers with respect to a Gaussian integer base $b$. 

The following result from \cite{KS74} provides conditions on a Gaussian integer $b$ to ensure that any complex number can be written with respect to $b$ where the coefficients of the expansion are chosen from the set $\{0, 1, \ldots, |b|^{2}-1\}$. This choice is in some sense canonical due to its similarity to the usual choice of digits when representing real numbers using an integer base. 

\begin{theorem} \label{thm:radixExistence}[I. Katai, J. Szabo, \cite{KS74}, theorem 2] %%EXISTENCE OF RADIX EXPANSIONS FOR B=-N+I
Suppose $n$ is a positive integer and set $b = -n+i$. Let $z$ be an element of $\mathbb{C}$. There exist coefficients $d_{k}\in\Lambda :=\{0, 1, \ldots, |b|^{2}-1\}$ and some integer $\ell$ such that 
\begin{equation}
z = d_{\ell}b^{\ell} + d_{\ell-1}b^{\ell-1} + \cdots + d_{0} + \sum_{k\geq1}d_{-k}b^{-k}.
\end{equation}
\end{theorem}
The expansions are called \textit{radix expansions}. The set $\Lambda$ implicitly depends on the base $b$. Our convention will be to not include $b$ in the notation. This is because we only consider a single Gaussian integer base at a time in all our discussions and wish to keep our notation simple. 

\begin{definition} %%RESTRICTED DIGIT CANTOR SETS IN THE PLANE
Let $b = -n+i$ where $n$ is a positive integer. Suppose $D$ is a nonempty subset of $\Lambda$. We call the set 
\begin{equation}
C_{D} := \bigg\{\sum_{k\geq1}d_{k}b^{-k}\in\mathbb{C} : d_{k}\in D\bigg\}
\end{equation}
the \textit{base-$b$ restricted digit Cantor set with digit set $D$}. 
\end{definition}

We again omit any indication of the base $b=-n+i$ for the same reason the base $b$ is omitted from the notation $\Lambda$. As in Section~\ref{section:realCase}, we maintain this terminology even when $D$ is a singleton or equal to $\Lambda$. The set $C_{D}$ is notably not a Cantor set in either of these extreme cases. 

Consider the following two facts about base-$b$ restricted digit Cantor sets. 
\begin{lemma}\label{lem:restrictedSimilar} %% RESTRICTED DIGIT CANTOR SETS IN C ARE SELF SIMILAR SETS
The base-$b$ restricted digit set $C_{D}$ is the attractor of the IFS $\mathcal{F} = \{z\mapsto\frac{z+d}{b}:d\in D\}$.
\end{lemma}

\begin{proof}
Since $|b| > 1$, the maps in $\mathcal{F}$ are contractions. The equation 
\begin{equation}
C_{D} = \cup_{f\in\mathcal{F}} f(C_{D})
\end{equation} 
can be verified directly. Let us explain why $C_{D}$ is compact. We first endow $D$ with the discrete topology. Observe that $D^\mathbb{N}$ with the product topology is a compact space. The coding map $\pi:D^{\mathbb{N}} \to C_{D}$ is equivalent to the map $(d_{k})_{k\geq1}\mapsto\sum_{k\geq1}d_{k}b^{-k}$ in this context. To see this, observe that for $f_{d_{1}}, f_{d_{2}}, \ldots, f_{d_{m}}\in\mathcal{F}$ we have
\begin{equation}
f_{d_{1}} \circ f_{d_{2}} \circ \cdots \circ f_{d_{m}}(0) = \sum_{k=1}^m d_{k}b^{-k}. 
\end{equation}
This map is both continuous and surjective. Since the image of a compact set under a continous map is compact, we see that $C_{D}$ is compact. Theorem \ref{thm:hutchinson} identifies $C_{D}$ as the attractor. 
\end{proof}

\begin{lemma} \label{lem:digitSep} %% CAN SEPARATE REPS USING DIGITS
Let $b = -n+i$ with $n \geq 2$ and suppose $D \subset \Lambda$ satisfies the condition that for all $d, d^{'}\in D$, we have $|d-d^{'}| \neq 1$. Every element of $C_{D}$ has a unique radix expansion that only uses digits in $D$. 
\end{lemma}

We postpone the proof of this lemma to after our definition for $\times b$-invariance (Definition~\ref{def:timesb}). We can now argue that the corresponding iterated function system of a restricted digit set with sufficiently separated digits satisfies the strong separation condition.

\begin{proposition} \label{prop:strongSeparation} %% WE GET A STRONG SEPARATION CONDITION
Let $b = -n+i$ where $n \geq 2$ and suppose $D \subset \Lambda$ is nonempty and satisfies the condition that for all $d, d^{'}\in D$, we have $|d-d^{'}| \neq 1$. The iterated function system $\mathcal{F}$ corresponding to $C_{D}$ satisfies the strong separation condition. That is, if $f_{1}, f_{2}\in\mathcal{F}$ are distinct, then $f_{1}(C_{D}) \cap f_{2}(C_{D}) = \emptyset$. 
\end{proposition}

\begin{proof}
Let $z, w$ be elements of $C_{D}$. By Lemma~\ref{lem:digitSep}, there exist unique radix expansions for $z$ and $w$ of the form 
\begin{align}
z = \sum_{k\geq1}a_{k}b^{-k} \\ 
w = \sum_{k\geq1}c_{k}b^{-k} 
\end{align}
respectively. The digits $a_{k}$ and $c_{k}$ are elements of $D$ for every $k$. Suppose $f_{1}, f_{2}$ are distinct maps in $\mathcal{F}$. There exists $d_{1}, d_{2}\in D$, with $d_{1} \neq d_{2}$, such that $f_{1}(z) = \frac{z+d_{1}}{b}$ and  $f_{2}(w) = \frac{w+d_{2}}{b}$. It follows that 
\begin{align}
f_{1}(z) = d_{1}b^{-1} + \sum_{k\geq1}a_{k}b^{-(k+1)}, \\ 
f_{2}(z) = d_{2}b^{-1} + \sum_{k\geq1}c_{k}b^{-(k+1)}.
\end{align}
Since $d_{1}$ and $d_{2}$ are also in $D$, it follows that these radix expansions are respectively the unique radix expansions for $f_{1}(z)$ and $f_{2}(z)$ that only use digits in $D$. The fact that $d_{1}\neq d_{2}$ ensures that the expansions are not the same and thus cannot represent that same complex number. We conclude that $f_{1}(C_{D}) \cap f_{2}(C_{D}) = \emptyset$. 
\end{proof}

Using the language of Definition~\ref{def:contractInvariant}, the sets $C_{D}$, when $D$ is sufficiently separated, contain $\times |b|^{-1}$-invariant sets and are $\times |b|^{-1}$-invariant sets themselves. The following definition classifies this particular case of invariance. 

\begin{definition} \label{def:timesb} %% TIMES B INVARIANCE
Let $b = -n+i$ with $n \geq 2$ and suppose $D \subset \Lambda$ is nonempty and satisfies the condition that for all $d, d^{'}\in D$, we have $|d-d^{'}| \neq 1$. A nonempty closed subset $Y\subset C_{D}$ is called $\times b$-\textit{invariant} if it is $\times |b|^{-1}$-invariant. 
\end{definition}

\begin{example} %% ANALOGY TO REAL CASE 
The restricted digit Cantor set $C_{D}$ is $\times b$-invariant if the digit set $D$ satisfies $|d-d^{'}|\neq1$ for all $d, d^{'}\in D$. 
\end{example}

This concludes what is required to state and prove our main theorem (Theorem~\ref{thm:main}). The remainder of this section presents the proof of Lemma~\ref{lem:digitSep}. To prove Lemma~\ref{lem:digitSep}, we first explain when a radix expansion of a complex number is not unique. In other words, when the preimage of a complex number under the coding map $\pi:D^{\mathbb{N}}\to C_{D}$ is not a singleton. We begin by introducing new notation. 
%% NOTATION FOR THE USE OF PROVING THE LEMMAS
It is convenient to use the notation 
\begin{equation}\label{eq:radix}
(d_{\ell},d_{\ell-1},\ldots,d_{0};d_{-1},\ldots)
\end{equation} 
for a radix expansion with digits $d_{k}\in D$. In the discussions that follow this always refers to an expansion in base $b=-n+i$. We use the notation $d_{\ell}d_{\ell-1}\cdots d_{0}.d_{-1}\cdots$ to denote the complex number $\sum_{k=-\infty}^{\ell}d_{k}b^{k}$ represented by (\ref{eq:radix}). The point that we would call the decimal point, if this was an expansion in base ten, is called the \textit{radix point}. We refer to the digits to the left of the radix point $(d_{\ell},d_{\ell-1},\ldots d_{0};)$ as the \textit{integer part} of the expansion. The complex number represented by the integer part of a radix expansion is the Gaussian integer $d_{\ell}b^{\ell} + d_{\ell-1}b^{\ell-1} + \cdots + d_{0}$.

Radix expansions of complex numbers, like expansions of real numbers in an integer base, are not unique. In fact, it is shown in \cite{G82} that there can be as many as three different radix expansions in the same base for the same complex number. A result of Gilbert in \cite{G82} places a necessary and sufficient condition on a pair of equivalent radix expansions. We require the following notation to state it. 

Let $p = (p_{\ell}, p_{\ell-1}, \ldots, p_{0};p_{-1},\ldots)$ be a radix expansion and let $k$ be an integer. We denote the Gaussian integer represented by the integer part of the radix expansion $(p_{\ell}, p_{\ell-1}, \ldots, p_{k};$ $p_{k-1},\ldots)$ by $p(k)$. 

\begin{lemma}\label{lem:lemstates} [W. J. Gilbert, \cite{G82}, proposition 1] %%ALLOWABLE STATES

Let $n$ be a postive integer. Two radix expansions, $q$ and $r$, represent the same complex number in base $b = -n+i$ if and only if, for all integers $k$, either

\begin{itemize}\label{lem:states}
\item[(i)] $q(k)-r(k) \in \{0, \pm1, \pm(n+i), \pm(n-1+i)\}$ when $n\neq 2$, or
\item[(ii)] $q(k)-r(k) \in \{0, \pm1, \pm(2+i), \pm(1+i), \pm i, \pm (2+2i)\}$ when $n = 2$. 
\end{itemize}
\end{lemma}

This lemma can be used to deduce what expansions are possible for complex numbers that have multiple radix expansions. It is also through this analysis that it can be shown that a complex number has at most three representations in base $b=-n+i$. We restrict ourselves to the case that $n\geq 2$. Enforcing that pairs of digits are not distance $1$ apart in the case of $b=-1+i$ implies that $C_{D}$ is a singleton.  
%%WON'T TREAT THE FENCING HERE

In \cite{G82}, Gilbert derives a state graph that governs triples of radix expansions that represent the same complex number. We present the exposition needed to derive and parse the graph. 

Suppose $p, q$ and $r$ are radix expansions of the same complex number. We do not assume that they are distinct. We define the $k$th state of $p, q$ and $r$ to be the triple 
\begin{equation}
S(k) := (p(k)-q(k), q(k)-r(k), r(k)-p(k)).
\end{equation}
Notably, since the sum of these components is zero, one of the components is redundant. Nonetheless, it is useful to express all the differences explicitly in order to determine the digits at the $k$th place of the expansions $p, q,$ and $r$. We describe how to do this now. 

If $p = (p_{\ell},p_{\ell-1},\ldots p_{0};p_{-1},\ldots)$, then $p(k+1)$ is the Gaussian integer with radix expansion $(p_{\ell},p_{\ell-1},\ldots,p_{k+1};)$. Therefore we have $p(k) = bp(k+1) + p_{k}$. It follows that $p(k) - q(k) = p_{k}-q_{k}+b(p(k+1)-q(k+1))$. We can capture this as a relationship between states with the equation  
\begin{equation}\label{eq:states}
S(k) = (p_{k}-q_{k}, q_{k}-r_{k}, r_{k}-p_{k}) + bS(k+1).
\end{equation}
Therefore the knowledge of the value of $S(k+1)$ can be used with Lemma~\ref{lem:states} to determine the possible values for the digits $p_{k}, q_{k}$, and $r_{k}$ and the state $S(k)$. 

If we treat allowable states as nodes, we can contruct the graph. The directed edges indicate what states $S(k)$ can be achieved from a given state $S(k+1)$ (the node you are currently at). The graph in Figure~\ref{fig:radix} corresponds to the cases $n\geq3$ where $b=-n+i$. The case $n=2$ is more complicated and is presented in Appendix~\ref{app:b}. Both graphs feature a system of diagrams that communicate the value of a state. We describe the system for the case $n\geq3$ here. The additional states present in the case $n=2$ can be found in Appendix~\ref{app:b}. 

We begin with a system of diagrams that communicate the value of $p(k)-q(k)$. The system is as follows:

\begin{enumerate}

\item[(i)]  $p(k)-q(k) = 0$ corresponds to \begin{tikzpicture}\draw (0,0) rectangle node{pq} (.75,.75); \end{tikzpicture}.
\item[(ii)] $p(k)-q(k) = 1$ corresponds to  \begin{tikzpicture}\draw (0,0) rectangle node{q} (0.75,0.75); \draw (0.75, 0.75) rectangle  node{p} (1.5, 0); \end{tikzpicture}.

\item[(iii)] $p(k)-q(k) = n-1+i$ corresponds to  \begin{tikzpicture}\draw (0,0) rectangle node{p} (0.75,0.75); \draw (0, 0) rectangle  node{q} (.75, -.75); \end{tikzpicture}.
\item[(iv)] $p(k)-q(k) = n+i$ corresponds to  \begin{tikzpicture}\draw (0,0) rectangle node{q} (0.75,0.75); \draw (0.75, 0.75) rectangle  node{p} (1.5, 1.5); \end{tikzpicture}.
\end{enumerate}

\begin{figure} [p!]%%GRAPH GOVERNING RADIX EXPANSIONS N>= 3
\centering
\vspace*{-3em}
\hspace*{-5.5em}
\begin{tikzpicture}
\begin{scope}[every node/.style]
     \node (A) at (0, 7) {\begin{tikzpicture}\draw (0,0) rectangle node{pqr} (.75,.75); \end{tikzpicture}}; %TOP NODE
    \node (B) at (-5.5,3.5) {\begin{tikzpicture}\draw (0,0) rectangle node{pq} (0.75,0.75); \draw (0.75, 0.75) rectangle  node{r} (1.5, 0); \end{tikzpicture}}; %LEFT OF A
    \node (C) at (-5.5,-3.5) {\begin{tikzpicture}\draw (0,0) rectangle node{pq} (0.75,0.75); \draw (0.75, 0.75) rectangle node{r} (0, 1.5); \end{tikzpicture}}; %DOWN FROM B
    \node (D) at (-2.5,0) {\begin{tikzpicture}\draw (0,0) rectangle node{r} (0.75,0.75); \draw (0.75, 0.75) rectangle node{pq} (1.5, 1.5); \end{tikzpicture}}; %UP-RIGHT FROM C
    \node (E) at (5.5,3.5) {\begin{tikzpicture}\draw (0,0) rectangle node{r} (0.75,0.75); \draw (0.75, 0.75) rectangle node{pq} (1.5, 0); \end{tikzpicture}}; %RIGHT OF A
    \node (F) at (5.5,-3.5) {\begin{tikzpicture}\draw (0,0) rectangle node{r} (0.75,0.75); \draw (0.75, 0.75) rectangle node{pq} (0, 1.5); \end{tikzpicture}}; %DOWN FROM E
    \node (G) at (2.5, 0) {\begin{tikzpicture}\draw (0,0) rectangle node{pq} (0.75,0.75); \draw (0.75, 0.75) rectangle node{r} (1.5, 1.5); \end{tikzpicture}}; %UP-LEFT FROM F

     \node (H) at (-7,-7) {\begin{tikzpicture}\draw (0,0) rectangle node{p} (0.75,0.75); \draw (0.75, 0.75) rectangle node{r}(0, 1.5);\draw (0, 0) rectangle node{q} (-0.75, 0.75); \end{tikzpicture}}; %TOP OF TRIPLE (LHS)
    \node (I) at (-7, -10) {\begin{tikzpicture}\draw (0,0) rectangle node{q} (0.75,0.75); \draw (0.75, 0.75) rectangle  node{p} (0, 1.5);\draw (0, 0) rectangle node{r} (-0.75, 0.75); \end{tikzpicture}}; %DOWN FROM H
    \node (J) at (-9.5,-8.5) {\begin{tikzpicture}\draw (0,0) rectangle node{r} (0.75,0.75); \draw (0.75, 0.75) rectangle node{q} (0, 1.5); \draw (0, 0) rectangle node{p} (-0.75, 0.75); \end{tikzpicture}}; %UP-LEFT FROM I
    \node (K) at (7, -7) {\begin{tikzpicture}\draw (0,0) rectangle node{r} (0.75,0.75);\draw (0.75, 0.75) rectangle node{p} (0, 1.5); \draw (0.75, 0.75) rectangle node{q} (1.5, 1.5); \end{tikzpicture}}; %TOP OF TRIPLE (RHS)
    \node (L) at (7, -10) {\begin{tikzpicture}\draw (0,0) rectangle node{p} (0.75,0.75); \draw (0.75, 0.75) rectangle node{q}(0, 1.5);  \draw (0.75, 0.75) rectangle node{r} (1.5, 1.5); \end{tikzpicture}}; %DOWN FROM K
    \node (M) at (9.5, -8.5) {\begin{tikzpicture}\draw (0,0) rectangle node{q} (0.75,0.75); \draw (0.75, 0.75) rectangle node{r} (0, 1.5);  \draw (0.75, 0.75) rectangle node{p} (1.5, 1.5); \end{tikzpicture}}; %UP-RIGHT FROM L
   
\end{scope}

\begin{scope}[>={Stealth[black]},
              every node/.style={fill=white},
              every edge/.style={draw=black}]
    \path [->] (A) edge[loop above]  node[above, fill=none]{$\scriptsize\begin{matrix}
0 \\
0 \\
0 \\
\end{matrix}$+} (A);
    \path [->] (A) edge node[above, fill=none]{$\scriptsize\begin{matrix}
0 \\
0 \\
1 \\
\end{matrix}$+}(B);
    \path [->] (A) edge node[above, fill=none]{$\scriptsize\begin{matrix}
1 \\
1 \\
0 \\
\end{matrix}$+}(E);
    \path [->] (B) edge node[left, fill=none]{$\scriptsize\begin{matrix}
1 \\
0 \\
2n \\
\end{matrix}$+}(H);
    \path [->] (E) edge node[right, fill=none]{$\scriptsize\begin{matrix}
2n-1 \\
2n \\
0 \\
\end{matrix}$+}(K);
    \path [->] (C) edge node[near end, below, yshift=-0.25cm,fill=none]{$\scriptsize\begin{matrix}
0 \\
1 \\
n^2-2n+2 \\
\end{matrix}$+}(K); 
    \path [->] (F) edge node[near end, below, yshift=-0.25cm, fill=none]{$\scriptsize\begin{matrix}
n^2-2n+2 \\
n^2-2n+1 \\
0 \\
\end{matrix}$+}(H);
    \path [->] (H) edge node[right, fill=none]{$\scriptsize\begin{matrix}
2n-1\\
0 \\
n^2 \\
\end{matrix}$}(I);
    \path [->] (I) edge node[below left, fill=none]{$\scriptsize\begin{matrix}
n^2 \\
2n-1 \\
0 \\
\end{matrix}$}(J);
    \path [->] (J) edge node[above, fill=none]{$\scriptsize\begin{matrix}
0 \\
n^2 \\
2n-1 \\
\end{matrix}$}(H);
    \path [->] (K) edge node[left, fill=none]{$\scriptsize\begin{matrix}
n^2-2n+1 \\
n^2 \\
0 \\
\end{matrix}$}(L);
    \path [->] (L) edge node[below right, fill=none]{$\scriptsize\begin{matrix}
0 \\
n^2-2n+1 \\
n^2 \\
\end{matrix}$}(M);
    \path [->] (M) edge node[above, fill=none]{$\scriptsize\begin{matrix}
n^2 \\
0 \\
n^2-2n+1 \\
\end{matrix}$}(K);
   
\end{scope}

\begin{scope}[>={Stealth[red]}, %%THE SUBGRAPH OF THE TAIL OF TWO RADIX EXPANSIONS
              every node/.style={fill=white},
              every edge/.style={draw=red, thick}]

    \path [->] (C) edge node[right, fill=none]{$\scriptsize\begin{matrix}
0 \\
0 \\
n^2-2n+1 \\
\end{matrix}$+}(D);
 \path [->] (B) edge node[near end, right, yshift=0.5cm, fill=none]{$\scriptsize\begin{matrix}
0 \\
0 \\
2n-1 \\
\end{matrix}$+}(C);
\path [->] (B) edge node[above, fill=none]{$\scriptsize\begin{matrix}
0 \\
0 \\
2n \\
\end{matrix}$+}(G);
 \path [->] (E) edge node[near end, left, yshift=0.5cm, fill=none]{$\scriptsize\begin{matrix}
2n-1 \\
2n-1 \\
0 \\
\end{matrix}$+}(F);
\path [->] (C) edge[transform canvas={yshift=1mm}] node[above, fill=none]{$\scriptsize\begin{matrix}
0 \\
0 \\
n^{2}-2n+2 \\
\end{matrix}$+} (F);
   \path [->] (F) edge[transform canvas={yshift=-1mm}] node[below, fill=none]{$\scriptsize\begin{matrix}
n^2-2n+2 \\
n^2-2n+2 \\
0 \\
\end{matrix}$+} (C);
 \path [->] (E) edge node[above, fill=none]{$\scriptsize\begin{matrix}
2n \\
2n \\
0 \\
\end{matrix}$+}(D);
 \path [->] (F) edge node[left, fill=none]{$\scriptsize\begin{matrix}
n^2-2n+1 \\
n^2-2n+1 \\
0 \\
\end{matrix}$+}(G);  
   \path [->] (G) edge node[below right, fill=none]{$\scriptsize\begin{matrix}
0 \\
0 \\
n^2 \\
\end{matrix}$}(E);
 \path [->] (D) edge node[below left, fill=none]{$\scriptsize\begin{matrix}
n^2 \\
n^2 \\
0 \\
\end{matrix}$}(B); 
\end{scope}
\end{tikzpicture}
\caption{The graph governing equivalent radix expansions in base $-n+i$ for $n\geq3$.}
\label{fig:radix}
\end{figure}

Swapping the positions of $p$ and $q$ in any of these arrangements flips the sign on the value of $p(k)-q(k)$. We can use this system to represent the mutual differences between $p(k), q(k)$ and $r(k)$ simultaneously. For example, the state $(1, -n-i, n-1+i)$ is communicated by 
$$
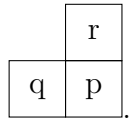
\begin{tikzpicture}\draw (0,0) rectangle node{p} (0.75,0.75); \draw (0.75, 0.75) rectangle node{r}(0, 1.5);\draw (0, 0) rectangle node{q} (-0.75, 0.75); \end{tikzpicture}.$$
Each edge of the state graph is labelled with a triple of integers. These indicate a combination of digits, read from top to bottom, that $p_{k}$, $q_{k}$, and $r_{k}$ can be in order for (\ref{eq:states}) to hold. The indication of a ``$+$" symbol means that we may add the integer $t$ to each of the values, where $t$ can be $0, 1,\ldots$ up to the largest integer for which all three of the listed numbers, when shifted by $t$, are less than or equal to $n^{2}=|b|^{2}-1$. Therefore the integers listed along the edges in the state graph communicate the distances between the digits at that index. 

\begin{theorem}\label{thm:rules} [W. J. Gilbert, \cite{G82}, theorem 5] %%RADIX EXPANSIONS RULES
Let $p, q$, and $r$ be three radix expansions in base $-n+i$ with $n\geq3$. These expansions represent the same complex number if and only if they can be obtained from an infinite path through the state graph in Figure~\ref{fig:radix} starting at state $(0, 0, 0)$, if necessary relabelling $p, q$ and $r$. 
\end{theorem}

We include the derivation of figure~\ref{fig:radix} in the Appendix \ref{app:a}. A similar theorem statement also holds for base $-2+i$ and is included in Appendix \ref{app:b} (Theorem~\ref{thm:rules2}). The descriptions that follow pertain to Figure~\ref{fig:radix}. %%PERHAPS USE A REF LABEL FOR APPENDIX AND THE THEOREM STATEMENT FOR N=2

If a complex number has a unique radix expansion in base $-n+i$, with $n\geq3$, then $p=q=r$ and this triple is perpetually in the state $(0, 0, 0)$. Complex numbers with precisely two distinct radix expansions correspond to paths that eventually exit the initial state $(0, 0, 0)$ but remain in the bolded red subgraph that does not distinguish between $p$ and $q$. Complex numbers with three distinct radix expansions eventually exit the initial state $(0, 0, 0)$ and ultimately are trapped in one of the two loops of period three at the bottom of the diagram.

We provide an example to illustrate how to read the graph. 
\begin{example}
The complex number $\frac{-23-10i}{17}$ has the following three radix expansions in base $b=-3+i$:
\begin{equation*}
\begin{split}
p &=(0;\overline{4,0,9,}) , \\
q &= (1;\overline{9,4,0,}), \\
r & = (1,5,5;\overline{0,9,4,}). \\
\end{split}
\end{equation*}
%%ABCK
The bar over the digits to the right of the radix point indicates a repetition of those digits with period three. The path that this number corresponds to in the state graph is the path that moves along the states

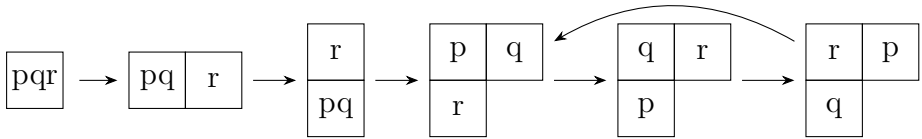
\begin{tikzpicture}
\begin{scope}[every node/.style]
    \node (A) at (0, 0) {\begin{tikzpicture}\draw (0,0) rectangle node{pqr} (.75,.75); \end{tikzpicture}};
    \node (B) at (2, 0) {\begin{tikzpicture}\draw (0,0) rectangle node{pq} (0.75,0.75); \draw (0.75, 0.75) rectangle  node{r} (1.5, 0); \end{tikzpicture}};
    \node (C) at (4, 0) {\begin{tikzpicture}\draw (0,0) rectangle node{pq} (0.75,0.75); \draw (0.75, 0.75) rectangle node{r} (0, 1.5); \end{tikzpicture}};
    \node (K) at (6, 0) {\begin{tikzpicture}\draw (0,0) rectangle node{r} (0.75,0.75);\draw (0.75, 0.75) rectangle node{p} (0, 1.5); \draw (0.75, 0.75) rectangle node{q} (1.5, 1.5); \end{tikzpicture}};
    \node (L) at (8.5, 0) {\begin{tikzpicture}\draw (0,0) rectangle node{p} (0.75,0.75); \draw (0.75, 0.75) rectangle node{q}(0, 1.5);  \draw (0.75, 0.75) rectangle node{r} (1.5, 1.5); \end{tikzpicture}};
    \node (M) at (11, 0) {\begin{tikzpicture}\draw (0,0) rectangle node{q} (0.75,0.75); \draw (0.75, 0.75) rectangle node{r} (0, 1.5);  \draw (0.75, 0.75) rectangle node{p} (1.5, 1.5); \end{tikzpicture}};
   
\end{scope}

\begin{scope}[>={Stealth[black]},
              every node/.style={fill=white},
              every edge/.style={draw=black}]
    \path [->] (A) edge (B);
    \path [->] (B) edge (C);
    \path [->] (C) edge (K);
    \path [->] (K)  edge (L);
    \path [->] (L)  edge (M);
    \path [->] (M) edge[bend right] (K); 
\end{scope}
\end{tikzpicture}.

This path also captures the complex number $\frac{-108+24i}{17} = 21.\overline{409} = 22.\overline{904} = 176.\overline{094}$. The distances between pairs of coefficients of the same power of $b$ is the same as those in the previous triple of expansions. 
\end{example}
A list of interesting observations about the state graph can be found in~\cite{G82}. We state an additional observation.

\begin{corollary}\label{cor:pm1} %%\pm1
Suppose $x$ and $y$ are two distinct radix expansions of the same complex number in base $-n+i$ where $n\geq2$. Let $k\in\mathbb{Z}$ be the first index at which a pair of digits $x_{k}$ and $y_{k}$ are not equal. Then $x_{k} - y_{k} = \pm 1$. 
\end{corollary}
\begin{proof}
The analysis that follows corresponds to the graph in Figure~\ref{fig:radix} governing radix expansions in base $-n+i$ for base $n\geq3$. A similar analysis can be done for the graph governing the case $n=2$ in Appendix~\ref{app:b} (Figure \ref{fig:radix2}). 

If $x$ and $y$ are the only distinct radix expansions of the complex number they represent, then they correspond to a path that, eventually, leaves the initial state $(0, 0, 0)$ and then remains in the bolded red subgraph of Figure~\ref{fig:radix}. Without loss of generality, we label $p = q = x$ and $r = y$. The first instance that an entry of $r$ differs from that of $p$ is when the path leaves the state $(0, 0, 0)$. From the graph, we see that the pair of digits between $r$ and $p$ differ by $\pm1$ at that index of the radix expansions. 

If $x$ and $y$ are two of three distinct radix expansions, then the path they correspond to ultimately enters, and is trapped, in one of the two loops of period three at the bottom of the diagram. If either $x$ or $y$ fit the role of $r$, then the expansions again differ for the first time when they leave state $(0, 0, 0)$. If neither $x$ or $y$ can be assigned the role of $r$, then, without loss of generality, let $x=p$ and $y=q$. The two expansions differ at a change of state that enters one of the two loops of period three. There are four of these edges and they all indicate that the digits of $p$ and $q$ differ by $\pm1$. 
\end{proof}

We now prove Lemma~\ref{lem:digitSep}. 

\begin{proof} [Proof of Lemma~\ref{lem:digitSep}] \label{pf:digitSep} 
Suppose $z\in C_{D}$. By definition, $z$ has a radix expansion $q$ that only uses digits in $D$. By corollary~\ref{cor:pm1}, any other radix expansion of $z$, if one exists, must use a digit that differs by $\pm1$ from a digit in $q$. The separation condition on $D$ implies that this digit must not be in $D$. It follows that $q$ is unique. 
\end{proof}

\section{\textbf{Proof of Main Result}} \label{section:mainResult} %% PROOF OF MAIN RESULT

We restate Theorem \ref{thm:earlymain}. 

%% MAIN RESULT
\begin{theorem}\label{thm:main} %%WE MIGHT AS WELL BE THOROUGH
Let $b=-n+i$ with $n\geq2$ and assume $D\subset\Lambda := \{0, 1,\ldots, |b|^{2}-1\}$ is nonempty and satisfies $|d-d^{'}| \neq 1$  for all $d, d^{'}\in D$. Let $\pi: D^{\mathbb{N}}\rightarrow C_{D}$ be the coding map. If $A\subset D^{\mathbb{N}}$ is a subshift, then 
\begin{itemize}
\item[(i)]$\pi(A) = \{\sum_{k=1}^{\infty}a_{k}b^{-k}:(a_{k})_{k\geq1}\in A\}$ is $\times b$-invariant, and
\item[(ii)]$\dim_{B}\pi(A) = \frac{\mathcal{E}(A)}{\log{|b|}}.$
\item[(iii)] If $Y\subset C_{D}$ is a $\times b$-invariant set, then $\dim_{H}Y = \dim_{B}Y.$
\end{itemize}
\end{theorem}

We split the proof into two parts. We first prove the two claims concerning subshifts of $D^{\mathbb{N}}$. We require the following lemmas. 

\begin{lemma}\label{lem:hausbound} %% BALLS CAN ONLY INTERSECT SO MANY M-TILES
Let $b=-n+i$ with $n\geq2$ and let $D\subset\Lambda$ be nonempty. Fix a positive integer $m$. There exists a bound, independent of $m$, on the number of $m$-tiles (defined by $\{z\mapsto\frac{z+d}{b}:d\in D\}$) that any ball with radius less than or equal to $|b|^{-m}\diam{C_{D}}$ intersects.
\end{lemma}
\begin{proof}
We remark that the diameter of an $m$-tile is $|b|^{-m}\diam{C_{D}}$. %This holds true for any translation of the $m$-tile by a Gaussian integer.

First consider the following. Let $\delta > 0$ and let $w\in\mathbb{C}$. By Theorem~\ref{thm:radixExistence}, the complex number $w$ is in the set $C_{\Lambda}+\zeta$ for some Gaussian integer $\zeta$. We claim that there exists a bound on the number of sets of the form $C_{D}+g$ that intersect $B_{\delta}(w)$, independent of $\zeta$. Suppose $\zeta = 0$. If $g+z\in B_{\delta}(w)$, then it follows that 
\begin{equation}
|g| < \delta/2 + |w-z| \leq \delta/2 + \diam{C_{\Lambda}}.
\end{equation}
Therefore there exists an integer $M$ such that at most $M$ sets of the form $C_{D}+g$ intersect $B_{\delta}(w)$. This bound holds for every $w \in C_{\Lambda}+\zeta$ for any $\zeta$. If it did not we could translate back to the origin and realize a contradiction. 

Let $M$ specifically be the maximum number of translates of $C_{D}$ by Gaussian integers that a ball of radius $\diam{C_{D}}$ can intersect. If a ball with radius less than or equal to $|b|^{-m}\diam{C_{D}}$ intersects more than $M$ $m$-tiles, then we can scale all the $m$-tiles and the ball by $b^{m}$ to obtain a ball of diameter less than or equal to $\diam{C_{D}}$ that intersects more than $M$ translates of $C_{D}$. It follows that $M$ is the desired bound. 
\end{proof}

A version of the following lemma was stated in \cite{PS21} (lemma 5.2). 

\begin{lemma} \label{lem:boxtile}%%BOX-COUNTING DIMENSION WITH M-TILE
Let Y be a nonempty subset of a restricted digit Cantor set $C_{D}$. For a fixed integer $m\geq1$, let $N_{m} (Y)$ denote the smallest number of $m$-tiles needed to cover Y. Then the box-counting dimension of Y exists if and only if $\lim_{m\rightarrow\infty}\frac{\log{N_{m}(Y)}}{m\log{|b|}}$ exists, and, if so, this limit is the box-counting dimension of Y.
\end{lemma}

\begin{proof}
Let $N_{\delta}(Y)$ be the smallest number of sets of diameter $\delta$ needed to cover $Y$. Let $K_{\delta}(Y)$ be the minimum number of closed balls needed to cover $Y$. If the box-counting dimension $\dim_{B}Y$ exists then it is known (section 3.1 in \cite{F90}) that
\begin{equation}
\dim_{B}Y = \lim_{\delta\to0}\frac{\log N_{\delta}(Y)}{-\log\delta} = \lim_{\delta\to0}\frac{\log K_{\delta}(Y)}{-\log\delta}.
\end{equation}
Consider the sequence $(\delta(m))_{m\geq1}$ with $\delta(m) = |b|^{-m}\diam{C_{D}}$ . It follows from their definitions that $N_{\delta(m)}(Y)\leq N_{m}(Y)$. Let $M$ be the bound from Lemma~\ref{lem:hausbound}. The intersection of any ball of radius $\delta(m)$ with $C_{D}$ can be covered by $M$ $m$-tiles. This implies that $N_{m}(Y)\leq MK_{\delta(m)}(Y)$. 
These inequalities yield
\begin{equation}
\frac{\log N_{\delta(m)}(Y)}{-\log\delta(m)} \leq \frac{\log N_{m}(Y)}{m\log|b|} \leq \frac{\log(MK_{\delta(m)}(Y))}{-\log\delta(m)}.
\end{equation}
If we assume that $\dim_{B}(Y)$ exists, then taking $m\to\infty$ implies that $\lim_{m\to\infty}\frac{\log N_{m}(Y)}{m\log|b|} = \dim_{B}Y$. 

To achieve the converse, let $\delta > 0$. Since we are interested in the limiting behaviour as $\delta\to0$, we can further assume $\delta < 1$ with no loss of generality. There exists $m(\delta)$ such that $|b|^{-m(\delta)}\diam{C_{D}}<\delta\leq|b|^{-m(\delta)+1}\diam{C_{D}}$. This means that an $m(\delta)$ tile can be covered by a single closed ball of radius $\delta$. In addition, we have the inequality $N_{m(\delta) + 1}(Y)\leq MK_{\delta}(Y)$ from the preceding discussion. This provides us with the inequalities
\begin{equation} \label{ineq:tileCounting}
\frac{\log[\frac{1}{M}N_{m(\delta)+1}(Y)]}{-\log\delta}\leq \frac{\log K_{\delta}(Y)}{-\log\delta} \leq \frac{\log N_{m(\delta)}(Y)}{-\log\delta}.
\end{equation}
The right hand side of the inequality is equal to $\frac{\log N_{m(\delta)}(Y)}{m(\delta)\log|b|} \frac{m(\delta)\log|b|}{-\log\delta}$. The limiting behaviour of this quantitiy as $\delta\to0$ (choose an arbitrary sequence) is equivalent to the limiting behaviour of $\frac{\log N_{m(\delta)}(Y)}{m(\delta)\log|b|}$ since 
\begin{equation}
1 \leq \frac{m(\delta)\log|b|}{-\log\delta} \leq \frac{m(\delta)}{m(\delta)-1}\Bigg(1-\frac{\log(\diam{C_{D}})}{\log\delta}\Bigg). 
\end{equation}
A similar analysis can be done with the left-hand side of (\ref{ineq:tileCounting}). Therefore $\dim_{B}Y$ exists if $\lim_{m\to\infty} \frac{\log N_{m}(Y)}{m\log|b|}$ does. 
\end{proof}

%% IN POSITION TO PROVE THE SUBSHIFT PART OF THE MAIN THEOREM

\begin{theorem}\label{thm:top} 
Let $b=-n+i$ with $n\geq2$ and assume $D\subset\Lambda$ is nonempty and satisfies $|d-d^{'}|\neq1$ that for any $d, d^{'}\in D$. Let $\pi:D^{N}\to C_{D}$ be the coding map. If $A\subset D^{\mathbb{N}}$ is a subshift, then 
\begin{itemize}
\item[(i)]$\pi(A)$ is $\times b$-invariant, and
\item[(ii)]$\dim_{B}\pi(A) = \frac{\mathcal{E}(A)}{\log{|b|}}.$
\end{itemize}
\end{theorem}
\begin{proof}
By Proposition~\ref{prop:strongSeparation}, it follows from the condition $|d-d^{'}|\neq1$ for all $d, d^{'}\in D$ that $\mathcal{F} = \{z\mapsto \frac{z+d}{b}: d\in D\}$ satisfies the strong separation condition. Lemma \ref{lem:restrictedSimilar} establishes $C_{D}$ as the attractor. By Lemma~\ref{lem:entropyIFS}, it follows that $\pi(A)\subset C_{D}$ is $\times |b|^{-1}$-invariant. This is precisely what we mean by $\times b$-invariance (Definition~\ref{def:timesb}). 
By combining the second part of Lemma~\ref{lem:entropyIFS} with Lemma~\ref{lem:boxtile} we obtain
\begin{equation}
\frac{\mathcal{E}(A)}{\log|b|} = \lim_{m\to\infty} \frac{\log{N_{m}}(g(A))}{m\log|b|} = \dim_{B}\pi(A).
\end{equation}
\end{proof}

%%  HAUSDORFF DIMENSION PART OF MAIN THEOREM

To complete this section, we show that the Hausdorff and box-counting dimensions of a $\times b$-invariant set are equal. We follow Furstenberg's strategy in \cite{F67} (proof of proposition III.1). 

It is known that $\dim_{H}V \leq \dim_{B}V$ for any $V\subset\mathbb{R}^{n}$ whenever the box-counting dimension exists (section 3.1 \cite{F90}). Given a $\times b$-invariant set $Y$, we need to show that $\dim_{B}Y \leq \dim_{H}Y$. The main idea is to take advantage of the correspondence between the $m$-tiles of $C_{D}$ and the cylinder sets in the coding space $D^{\mathbb{N}}$. To this end, we express the claim $\dim_{B}Y \leq \dim_{H}Y$ in terms of $m$-tiles. This is accomplished by Lemma \ref{lem:leqHaus} and Lemma \ref{lem:semigroupTrans}. 

\begin{lemma} \label{lem:leqHaus} %% BOUNDING HAUSDORFF DIMENSION FROM BELOW
Let $b=-n+i$ with $n\geq2$  and assume $D\subset\Lambda$ is nonempty and satisfies $|d-d^{'}|\neq1$ for any $d, d^{'}\in D$. Let $Y\subset C_{D}$ be a $\times b$-invariant set. The inequality $\dim_{B}Y \leq \dim_{H}Y$ holds if the following implication holds: whenever $Y\subset \cup_{k=1}^{K}T_{k}$ and $s < \dim_{B}Y$, where $\{T_{k}\}_{k=1}^{K}$ is a finite collection of $m_{k}$-tiles, we have $\sum_{k=1}^{K}|b|^{-sm_{k}} \geq 1$.
\end{lemma}

\begin{proof}
Let $\{S_{k}\}$ be a $\delta$-cover of $Y$ by $m_{k}$-tiles. The set $\pi^{-1}(Y)$ is a closed subset of the compact space $D^{\mathbb{N}}$ and thus it is compact. The one-to-one correspondence between cylinder sets and $m$-tiles allows us to extract a finite subcover $\{T_{k}\}_{k=1}^{K}$ from $\{S_{k}\}$. By assumption, if $s < \dim_{B}Y$, we have $\sum_{k=1}^{K}|b|^{-sm_{k}} \geq 1$. Therefore $\sum_{k=1}^{\infty}(\diam S_{k})^{s} \geq (\diam{C_{D}})^{s} > 0$.

It follows that $\mathcal{T}^{s}(Y) > 0$ (recall Definition \ref{def:tileMeasure}). By Lemma \ref{lem:hausbound} there is a uniform bound on the number of $m$-tiles a ball of radius less than or equal to $|b|^{-m}\diam{C_{D}}$ can intersect. This means we satisfy the conditions needed to apply Lemma \ref{lem:francescowasright}. This lemma allows us to use $\mathcal{T}^{s}(Y)$ in place of $\mathcal{H}^{s}(Y)$. It follows that $s\leq \dim_{H}Y$. Since we can do this for every $s<\dim_{B}Y$, it must be that $\dim_{B}Y\leq\dim_{H}Y$. 
\end{proof}

%% RECOLLECTION FROM SECTION 1
Let us recall a construction from Section \ref{section:realCase}. Let $L = \cup_{n\geq1}D^{n}$. This is the set of all finite words that can be written using the digits in $D$. Given a subshift $A\subset D^{\mathbb{N}}$, let $R$ denote the subset of $L$ containing those tuples which occur as finite subwords of sequences in $A$. The set $\mathcal{L}_{n}(A)$ can be viewed as the elements of $R$ of length $n$. The set $R$ is a semigroup under concatenation. Let us say that a word $\rho$ \textit{divides} a word $\rho^{'}$ if $\rho^{'} = \rho\rho_{1}$ for some $\rho_{1}\in L$. 

\begin{lemma} \label{lem:semigroupTrans} %% EXPRESSING THE PROBLEM IN TERMS OF THE SEMIGROUP R
Let $b=-n+i$ with $n\geq2$. Let $Y\subset C_{D}$ be a $\times b$-invariant set. The following implications are equivalent. 
\begin{itemize}
\item[(i)] If $Y\subset \cup_{k=1}^{K}T_{k}$ and $s < \dim_{B}Y$, where $\{T_{k}\}$ is a finite collection of $m_{k}$-tiles, then $\sum_{k=1}^{K}|b|^{-sm_{k}} \geq 1$.
\item[(ii)] If there exists a finite collection $\{\rho_{k}\}_{k=1}^{K}\subset R$ such that whenever $\rho\in R$ is of sufficient length, there exists $k=1,2,\ldots, K$ such that $\rho_{k}$ divides $\rho$ and $s < \dim_{B}Y$, then $\sum_{k=1}^{K}|b|^{-sl(\rho_{k})} \geq 1$.
\end{itemize}
\end{lemma}

\begin{proof}
To see that $\dim_{B}Y$ exists, observe that Lemma \ref{lem:entropyIFS} implies that $A_{Y} := \pi^{-1}(Y) \subset D^{\mathbb{N}}$ is a subshift satisfying $\pi(A_{Y}) = Y$ where $\pi$ is the coding map. The existence of $\dim_{B}Y$ then follows from Theorem \ref{thm:top}. 

Suppose (i) holds and that $W = \{\rho_{k}\}_{k=1}^{K}\subset R$ satisfies the division propert in (ii).  For each $\rho_{k}= d_{1}^{(k)}d_{2}^{(k)}\cdots d_{m_{k}}^{(k)}\in W$, let $T_{k} = (C_{D})_{d_{1}^{(k)}, d_{2}^{(k)},\ldots, d_{m_{k}}^{(k)}}$. Note that $m_{k} = l(\rho_{k})$. For each $y\in Y$, there is a sequence $(y_{k})_{k\geq1}\subset D^{\mathbb{N}}$ such that $y = 0.y_{1}y_{2}\ldots$ since $Y\subset C_{D}$. Suppose $N$ is large enough that $y_{1}y_{2}\cdots y_{N}$ is divisible by some $\rho_{k}\in W$. It follows by the contruction of $T_{k}$ that $y\in T_{k}$ and thus $Y\subset\cup_{k=1}^{K}T_{k}$. 

If $s<\dim_{B}Y$, then $\sum_{k=1}^{K}|b|^{-sl(\rho_{k})} = \sum_{k=1}^{K}|b|^{-sm_{k}}\geq1$. Therefore (i) implies (ii). The proof for the other direction is similar and so we omit it. 
\end{proof}

\begin{theorem} \label{thm:dimEq} %% HAUSDORFF DIM = BOX COUNTING DIM
Let $b=-n+i$ with $n\geq2$ and assume $D\subset\Lambda$ is nonempty and satisfies $|d-d^{'}|\neq1$ for all $d, d^{'}\in D$.
If $Y\subset C_{D}$ is a $\times b$-invariant set, then $\dim_{H}Y = \dim_{B}Y.$
\end{theorem}

\begin{proof}
 It is known that $\dim_{H}E \leq \dim_{B}E$ for any $E\subset\mathbb{R}^{n}$ whenever the box-counting dimension exists (see \cite{F90}). Recall that Lemma \ref{lem:entropyIFS} implies that $A_{Y} := \pi^{-1}(Y) \subset D^{\mathbb{N}}$ is a subshift satisfying $\pi(A_{Y}) = Y$ where $\pi$ is the coding map. The existence of $\dim_{B}Y$ then follows from Theorem \ref{thm:top}. 

What is left to show is that $\dim_{B}Y \leq \dim_{H}Y$. Let $s < \dim_{H}Y$. Our strategy is to satisfy the conditions of Lemma \ref{lem:leqHaus}. The condition that has to be shown is that whenever $Y\subset \cup_{k=1}^{K}T_{k}$, where $\{T_{k}\}$ is a finite collection of $m_{k}$-tiles, we have $\sum_{k=1}^{K}|b|^{-sm_{k}}\geq1$. By way of contradiction, let $\{T_{k}\}$ be a finite cover of $Y$ by $m_{k}$-tiles and suppose $\sum_{k=1}^{K}|b|^{-sm_{k}}<1$. 

By Lemma \ref{lem:semigroupTrans} this is equivalent to the statement that there exists $\{\rho_{k}\}\subset R$ such that whenever $\rho\in R$ is of sufficient length, it is divisible by $\rho_{k}$ for some $k$, but $\sum_{k=1}^{K}|b|^{-sl(\rho_{k})}<1$. 

It follows from Lemma \ref{lem:semigroup} that $\sum_{\rho\in R}|b|^{-sl(\rho)}$ converges. 

On the other hand, Theorem \ref{thm:top} tells us that $\dim_{B}Y$ is equal to $\frac{\mathcal{E}(A_{Y})}{\log{|b|}}$ where $\mathcal{E}(A_{Y})$ is the topological entropy of $A_{Y}$. Since $s < \dim_{B}Y$, we have $s < \frac{\log{|\mathcal{L}_{m}(A_{Y})|}}{m\log{|b|}}$ for all sufficiently large $m$. Therefore $|\mathcal{L}_{m}(A_{Y})||b|^{-sm} > 1$ for all sufficiently large $m$. 

Observe that we can rewrite the sum $\sum_{R}|b|^{-sl(\rho)}$ by enumerating over the elements of $R$ in increasing length to obtain $\sum_{m=1}^{\infty}|\mathcal{L}_{m}(A_{Y})||b|^{-sm}$. The latter series diverges to $+\infty$. We have established our contradiction. If $s < \dim_{B}Y$, then we must have $\sum_{k=1}^{K}|b|^{-sl(\rho_{k})} \geq 1$.

Through the equivalence stated in Lemma \ref{lem:semigroupTrans}, we have satisfied the conditions of Lemma \ref{lem:leqHaus}. This completes the proof. 
\end{proof}

\section{\textbf{Results Concerning The Dimension of $C_{D}\cap (C_{D}+z)$}} \label{section:intersectionResults}

In this final section we discuss an application of Proposition \ref{prop:strongSeparation} to compute the box-counting dimension of an intersection of a restricted digit Cantor set with a translate of itself. 

Let $C_{D}$ be a base-$b$ restricted digit Cantor set and consider the set $F = \{z\in\mathbb{C}: C_{D} \cap (C_{D} + z)\neq \emptyset \}$. Given $\beta\in[0,1]$, we define $F_{\beta}$ to be those $z\in F$ for which $\dim_{B}(C_{D} \cap (C_{D}+z)) = \beta\dim_{B}C_{D}$. Pedersen and Shaw showed in \cite{PS21} that $F_{\beta}$ is dense in $F$ under certain conditions on the digit set $D$. We quote the primary results. Recall that while the box-counting dimension does not always exist, the lower box counting dimension always exists. We denote the lower box-counting dimension by $\underline{\dim}_{B}$. It is defined by replacing the limit in Definition \ref{def:boxcounting} with the lower limit. 

\begin{theorem}\label{thm:extremeSep}[S. Pedersen, V. Shaw, \cite{PS21}, theorem 7.4] %% PEDERSEN AND SHAW ORIGINAL RESULT
Let $b = -n+i$ with $n\geq2$. Suppose that $D\subset\Lambda$ satisfies $d\leq n^{2}/2$ for all $d\in D$ and $|a-a^{'}|\geq n+1$ for all distinct pairs $a, a^{'}\in\Delta := D-D$. We have
\begin{equation}
\underline{\dim}_{B}(C_{D} \cap (C_{D}+z)) = \liminf_{m\to\infty}\frac{\log G_{m}(z)}{m\log|b|}
\end{equation}
where 
\begin{equation}
G_{m}(z) := |D\cap(D+z_{1})||D\cap(D+z_{2})|\cdots|D\cap(D+z_{m})|
\end{equation}
and $z = 0.z_{1}z_{2}\ldots$ with $z_{k}\in\Delta$. 
\end{theorem}

\begin{corollary} \label{cor:denseTranslates}[S. Pedersen, V. Shaw, \cite{PS21}, corollary 7.5]
Assume the hypotheses of Theorem \ref{thm:extremeSep}. The set $F_{\beta}$ is dense in $F$ for any $\beta\in[0, 1]$. 
\end{corollary}

The separation condition $|a-a^{'}|\geq n+1$ imposed on $\Delta$ is used to ensure that the sets $0.a_{1}a_{2}\ldots a_{m} + b^{-m}C_{D}$ and $0.a_{1}^{'}a_{2}^{'}\ldots a_{k}^{'} + b^{-k}C_{D}$ where $a_{k}\in \Delta$ are disjoint if $a_{k}\neq a_{k}^{'}$ for at least one $k$ (\cite{PS21}, lemma 7.1). Using Proposition \ref{prop:strongSeparation}, we can achieve this with the separation condition $|a-a^{'}|\neq1$. 

\begin{lemma} \label{lem:minimalDisjointCondition}
Let $b = -n+i$ with $n\geq2$. Suppose that $D\subset\Lambda$ satisfies $d\leq n^{2}/2$ for all $d\in D$ and $|a-a^{'}|\geq n+1$ for all distinct pairs $a, a^{'}\in\Delta$. Suppose $a_{k}, a_{k}^{'}\in\Delta$ for $k=1,2,\ldots, k$ and $a_{k}\neq a_{k}^{'}$ for at least one $k$ . Then the sets $(C_{D})_{a_{1}, a_{2}, \ldots, a_{m}}$ and $ (C_{D})_{a_{1}^{'}, a_{2}^{'}, \ldots, a_{m}^{'}}$ are disjoint. 
\end{lemma}
\begin{proof}
We follow the strategy in the proof of lemma 7.1 in \cite{PS21}. Let $d_{\text{max}}$ be the maximum element of $D$. Let $E = \Delta + d_{\text{max}}$, $e_{k} = a_{k} + d_{\text{max}}$, and $e_{k}^{'} + d_{\text{max}}$. Observe that since $d\leq n^{2}/2$, we have $E\subset\{0, 1, \ldots, n^{2}\}$. Furthermore, since the new digits $e_{k}$ and $e_{k}^{'}$ are simply translations of elements in $\Delta$, the set $E$ satisfies the separation condition $|e-e^{'}|\neq 1$ for any pair $e, e^{'}\in E$. A direct application of Proposition \ref{prop:strongSeparation} yields that $(C_{D})_{e_{1}, e_{2}, \ldots, e_{m}}  \cap (C_{D})_{e_{1}^{'}, e_{2}^{'}, \ldots, e_{m}^{'}} = \emptyset$
where the contractions are given by $z\mapsto\frac{z+e_{k}}{b}$. If we translate these disjoint images by $0.d_{1}d_{2}\cdots d_{m}$ where $d_{k} := -d_{\text{max}}$ for each $k$, then we obtain the desired result. 
\end{proof}

Lemma \ref{lem:minimalDisjointCondition} can used to extend Theorem \ref{thm:extremeSep}, and consequently Corollary \ref{cor:denseTranslates}, to a larger class of digit sets $D$. 

\begin{theorem}\label{thm:minimalSep} %% PEDERSEN AND SHAW EXTENDED RESULT
Let $b = -n+i$ with $n\geq2$. Suppose that $D\subset\Lambda$ satisfies $d\leq n^{2}/2$ for all $d\in D$ and $|a-a^{'}|\geq n+1$ for all distinct pairs $a, a^{'}\in\Delta$. We have
\begin{equation}\label{eq:lowerIntersectionDim}
\underline{\dim}_{B}(C_{D} \cap (C_{D}+z)) = \liminf_{m\to\infty}\frac{\log G_{m}(z)}{m\log|b|}
\end{equation}
where 
\begin{equation}
G_{m}(z) := |D\cap(D+z_{1})||D\cap(D+z_{2})|\cdots|D\cap(D+z_{m})|
\end{equation}
and $z = 0.z_{1}z_{2}\ldots$ with $z_{k}\in\Delta$. 
\end{theorem}

\begin{corollary} \label{cor:moreDenseTranslates}
Assume the hypotheses of Theorem \ref{thm:minimalSep}. The set $F_{\beta}$ is dense in $F$ for any $\beta\in[0, 1]$. 
\end{corollary}

%% THE ORIGINAL PAPER LEFT BETA = 0 AND BETA = 1 AS EXERCISES (USE A TAIL OF MAXIMAL DIGITS OR A TAIL OF ZEROS)-ALSO NEED TO FORCE THEIR NOTATION TO ALIGN WITH MY CONVENTION
These results follow from Lemma \ref{lem:minimalDisjointCondition} as Theorem \ref{thm:extremeSep} and Corollary \ref{cor:denseTranslates} do in \cite{PS21} when the separation between distinct elements of $\Delta$ is at least $n+1$. For this reason we only sketch the proofs of Theorem~\ref{thm:extremeSep} and Corollary~\ref{cor:moreDenseTranslates}. 

In order for $C_{D}\cap (C_{D} + z)$ to be nonempty, it must be that $z = \sum_{k\geq1}z_{k}b^{-k}$ where $z_{k}\in\Delta$. For a fixed sequence $(z_{1}, z_{2}, \ldots)$, we use the notation $\lfloor z\rfloor_{m}$ to denote the truncation $0.z_{1}z_{2}\ldots z_{m}$. Lemma \ref{lem:minimalDisjointCondition} can be used to establish that $C_{D}\cap (C_{D} + \lfloor z\rfloor_{m}) = \{0.u_{1}u_{2}\ldots u_{m}:u_{k}\in D\cap (D+z_{k})\} + b^{-k}C_{D}$ and that $C_{D}\cap (C_{D} + z) \subset C_{D}\cap (C_{D} + \lfloor z\rfloor_{m})$ for all $m\geq1$. We then see that the $m$-tiles that cover $C_{D}\cap (C_{D} + z)$ are those whose specified digits lie in $D\cap(D+z_{k})$ for $1\leq k\leq m$. In fact, we have $N_{m}(C_{D}\cap(C_{D}+z)) = G_{m}(z)$, which implies (\ref{eq:lowerIntersectionDim}).

Let us describe how to construct points in $F_{\beta}$ that are arbitrarily close to some point in $F$. Let $0 \leq \beta \leq 1$ and let $z\in F$. Fix a sequence $(z_{1}, z_{2}, \ldots)$ such that $z = 0.z_{1}z_{2}\ldots$ and $z_{k}\in\Delta$. Given $r>0$, there exists $m$ such that any point of the form $w = 0.z_{1}z_{2}\ldots z_{m}w_{m+1}w_{m+2}\ldots$ is within $r$ distance of $z$. The choices made for the tail of the expansion defining $w$ depends on the value of $\beta$. 

If $0 < \beta < 1$, then we choose a sequence of positive integers $(h_{k})_{k\geq1}$ such that $h_{k}\leq k\beta < 1 + h_{k}$. It can be shown that either $h_{k+1} = h_{k}$ or $h_{k+1} = 1+ h_{k}$. For all $k>m$ we set
\begin{equation}
y_{k} = \begin{cases}
d_{\text{max}} - d_{\text{min}} & \;\text{if}\; h_{k} = h_{k-1},\\
0 & \;\text{if}\; h_{k} = 1 + h_{k-1},
\end{cases}
\end{equation}
where $d_{\text{max}}$ and $d_{\text{min}}$ are respectively the maximum and minimum of $D$. If $\beta = 0$ or if $\beta = 1$, respectively choose $y_{k} = d_{\text{max}} - d_{\text{min}}$ or $y_{k} = 0$ for all $k>m$. The condition that $y\in F_{\beta}$ can be checked directly using Theorem \ref{thm:minimalSep}. 

We conclude with an example illustrating Theorem~\ref{thm:minimalSep} for a case not covered by Theorem~\ref{thm:extremeSep}. 
\begin{example}
Let $b=-3+i$. The subsets  $D$ of $\{0, 1,\ldots, 9\}$ for which $d\leq9/2$ for all $d\in D$ are all subsets of $\{0, 1, 2, 3, 4\}$. Under the additional constraint that every pair of elements $a, a^{'}\in\Delta$ satisfies $a-a^{'}\neq1$, then the choices of $D$ with more than one element are $\{0, 2\}$, $\{0, 3\}$, $\{0, 4\}$, $\{1, 3\}$, $\{1, 4\}$, $\{2, 4\}$, and $\{0, 2, 4\}$. If $D$ is only a singleton, then so is $C_{D}$. The box-counting dimension of $C_{D}$ is zero in that case.  

Suppose we choose $D = \{0, 3\}$. Therefore we have $\Delta = \{-3, 0, 3\}$. The intersection of $C_{D}$ and its translation by $z = \frac{21-18j}{-19+26j} = 0.\overline{-303}$ is nonempty. Since $|D\cap(D-3)| = |D\cap(D+3)| = 1$ and $|D| = 2$, it follows that
\begin{equation}
G_{m}(z) = \begin{cases} 2^{m/3} &\;\text{if}\;m \equiv 0 \mod{3},\\
 2^{(m-1)/3} &\;\text{if}\;m \equiv 1 \mod{3},\\
 2^{(m+1)/3} &\;\text{if}\;m \equiv 2 \mod{3}.
\end{cases}
\end{equation}
Therefore $2^{(m-1)/3} \leq G_{m}(z) \leq 2^{(m+1)/3}$ for all $m$. In particular, 
\begin{equation}
\frac{2(m-1)\log{2}}{3m\log{10}} \leq \frac{\log{G_{m}(z)}}{m\log{10}}\leq \frac{2(m+1)\log{2}}{3m\log{10}}.
\end{equation}
By Theorem~\ref{thm:minimalSep}, we conclude that $\dim_{B}(C_{D}\cap(C_{D} + z)) = \frac{2\log{2}}{3\log{10}}$.
\end{example}
The number of sets $D$ that satisfy the conditions of Theorem~\ref{thm:minimalSep} and have more than one element is $7$ when $b=-3+i$. For this same base $b$,  the only non-singleton set $C_{D}$ that Theorem~\ref{thm:extremeSep} makes a statement about is the one corresponding to $D = \{0, 4\}$. The gap between the number of sets $C_{D}$ that are captured by both theorems increases quickly as $n$ increases ($b = -n+i$). We list some numerical data for small $n$ in Table~\ref{tbl:gapSize}.

\begin{table} 
\setlength{\arrayrulewidth}{0.5mm}
\setlength{\tabcolsep}{18pt}
\renewcommand{\arraystretch}{1.5}
\center
\begin{tabular}{|c|c|c|}

 \multicolumn{3}{c}{Number of nontrivial sets $C_{D}$ captured by Theorems~\ref{thm:extremeSep} and~\ref{thm:minimalSep}}\\
 \hline
$n$& Theorem~\ref{thm:extremeSep}&Theorem~\ref{thm:minimalSep}\\
 \hline
 2 & 0 & 1 \\
 3 & 1 & 7 \\ 
 4 & 10 & 58 \\ 
 5 & 29 & 300 \\
 6 & 87 & 2561 \\
 7 & 181 & 19004 \\
 \hline
\end{tabular}
%\vspace
\caption{$b=-n+i$ and $|D| > 1$}
\label{tbl:gapSize}
\end{table}

\nocite{*}
\renewcommand{\bibname}{References}

\bibliographystyle{plain} %unsrt
\clearpage
\bibliography{revisedReferences}

\appendix %SOME DETAILS CONCERNING THE STATE GRAPH
\section{Derivation of the State Graph ($n\geq3$)}\label{app:a}%%APPENDIX A
%%BUILD THE N=3 Graph
This appendix is a supplement to the discussion of Figure~\ref{fig:radix} in Section~\ref{section:complexCase}. The goal of this appendix is to demonstrate how Lemma~\ref{lem:statesrepeat} translates to the state graph in Figure~\ref{fig:radix}. For convenience, the graph can be found in Figure~\ref{fig:radixrepeat} below and Lemma~\ref{lem:statesrepeat} is simply a repetition of Lemma~\ref{lem:lemstates}. 

Recall that the claim is that any triple of radix expansions in base-(-n+i) represent the same complex number if and only if they can be obtained from an infinite path through the state graph starting from the top node (state). The diagrams for the states and the labelling system for the edges is the same as it is in Section~\ref{section:complexCase}. Given a radix expansion $(d_{\ell}, d_{\ell-1}, \ldots, d_{0};d_{-1}, d_{-2}, \ldots)$, we use the notation $d_{k}$ for the $k$th digit. The notation $d(k)$ is the same as it is in Section~\ref{section:complexCase}, but recalling it is unnecessary. Its meaning can be ignored in the context of the derivation of the state graph. 

\begin{figure} [p!]%%GRAPH GOVERNING RADIX EXPANSIONS N>= 3
\centering
\vspace*{-3em}
\hspace*{-5.5em}
\begin{tikzpicture}
\begin{scope}[every node/.style]
     \node (A) at (0, 7) {\begin{tikzpicture}\draw (0,0) rectangle node{pqr} (.75,.75); \end{tikzpicture}}; %TOP NODE
    \node (B) at (-5.5,3.5) {\begin{tikzpicture}\draw (0,0) rectangle node{pq} (0.75,0.75); \draw (0.75, 0.75) rectangle  node{r} (1.5, 0); \end{tikzpicture}}; %LEFT OF A
    \node (C) at (-5.5,-3.5) {\begin{tikzpicture}\draw (0,0) rectangle node{pq} (0.75,0.75); \draw (0.75, 0.75) rectangle node{r} (0, 1.5); \end{tikzpicture}}; %DOWN FROM B
    \node (D) at (-2.5,0) {\begin{tikzpicture}\draw (0,0) rectangle node{r} (0.75,0.75); \draw (0.75, 0.75) rectangle node{pq} (1.5, 1.5); \end{tikzpicture}}; %UP-RIGHT FROM C
    \node (E) at (5.5,3.5) {\begin{tikzpicture}\draw (0,0) rectangle node{r} (0.75,0.75); \draw (0.75, 0.75) rectangle node{pq} (1.5, 0); \end{tikzpicture}}; %RIGHT OF A
    \node (F) at (5.5,-3.5) {\begin{tikzpicture}\draw (0,0) rectangle node{r} (0.75,0.75); \draw (0.75, 0.75) rectangle node{pq} (0, 1.5); \end{tikzpicture}}; %DOWN FROM E
    \node (G) at (2.5, 0) {\begin{tikzpicture}\draw (0,0) rectangle node{pq} (0.75,0.75); \draw (0.75, 0.75) rectangle node{r} (1.5, 1.5); \end{tikzpicture}}; %UP-LEFT FROM F

     \node (H) at (-7,-7) {\begin{tikzpicture}\draw (0,0) rectangle node{p} (0.75,0.75); \draw (0.75, 0.75) rectangle node{r}(0, 1.5);\draw (0, 0) rectangle node{q} (-0.75, 0.75); \end{tikzpicture}}; %TOP OF TRIPLE (LHS)
    \node (I) at (-7, -10) {\begin{tikzpicture}\draw (0,0) rectangle node{q} (0.75,0.75); \draw (0.75, 0.75) rectangle  node{p} (0, 1.5);\draw (0, 0) rectangle node{r} (-0.75, 0.75); \end{tikzpicture}}; %DOWN FROM H
    \node (J) at (-9.5,-8.5) {\begin{tikzpicture}\draw (0,0) rectangle node{r} (0.75,0.75); \draw (0.75, 0.75) rectangle node{q} (0, 1.5); \draw (0, 0) rectangle node{p} (-0.75, 0.75); \end{tikzpicture}}; %UP-LEFT FROM I
    \node (K) at (7, -7) {\begin{tikzpicture}\draw (0,0) rectangle node{r} (0.75,0.75);\draw (0.75, 0.75) rectangle node{p} (0, 1.5); \draw (0.75, 0.75) rectangle node{q} (1.5, 1.5); \end{tikzpicture}}; %TOP OF TRIPLE (RHS)
    \node (L) at (7, -10) {\begin{tikzpicture}\draw (0,0) rectangle node{p} (0.75,0.75); \draw (0.75, 0.75) rectangle node{q}(0, 1.5);  \draw (0.75, 0.75) rectangle node{r} (1.5, 1.5); \end{tikzpicture}}; %DOWN FROM K
    \node (M) at (9.5, -8.5) {\begin{tikzpicture}\draw (0,0) rectangle node{q} (0.75,0.75); \draw (0.75, 0.75) rectangle node{r} (0, 1.5);  \draw (0.75, 0.75) rectangle node{p} (1.5, 1.5); \end{tikzpicture}}; %UP-RIGHT FROM L
   
\end{scope}

\begin{scope}[>={Stealth[black]},
              every node/.style={fill=white},
              every edge/.style={draw=black}]
    \path [->] (A) edge[loop above]  node[above, fill=none]{$\scriptsize\begin{matrix}
0 \\
0 \\
0 \\
\end{matrix}$+} (A);
    \path [->] (A) edge node[above, fill=none]{$\scriptsize\begin{matrix}
0 \\
0 \\
1 \\
\end{matrix}$+}(B);
    \path [->] (A) edge node[above, fill=none]{$\scriptsize\begin{matrix}
1 \\
1 \\
0 \\
\end{matrix}$+}(E);
    \path [->] (B) edge node[left, fill=none]{$\scriptsize\begin{matrix}
1 \\
0 \\
2n \\
\end{matrix}$+}(H);
    \path [->] (E) edge node[right, fill=none]{$\scriptsize\begin{matrix}
2n-1 \\
2n \\
0 \\
\end{matrix}$+}(K);
    \path [->] (C) edge node[near end, below, yshift=-0.25cm,fill=none]{$\scriptsize\begin{matrix}
0 \\
1 \\
n^2-2n+2 \\
\end{matrix}$+}(K); 
    \path [->] (F) edge node[near end, below, yshift=-0.25cm, fill=none]{$\scriptsize\begin{matrix}
n^2-2n+2 \\
n^2-2n+1 \\
0 \\
\end{matrix}$+}(H);
    \path [->] (H) edge node[right, fill=none]{$\scriptsize\begin{matrix}
2n-1\\
0 \\
n^2 \\
\end{matrix}$}(I);
    \path [->] (I) edge node[below left, fill=none]{$\scriptsize\begin{matrix}
n^2 \\
2n-1 \\
0 \\
\end{matrix}$}(J);
    \path [->] (J) edge node[above, fill=none]{$\scriptsize\begin{matrix}
0 \\
n^2 \\
2n-1 \\
\end{matrix}$}(H);
    \path [->] (K) edge node[left, fill=none]{$\scriptsize\begin{matrix}
n^2-2n+1 \\
n^2 \\
0 \\
\end{matrix}$}(L);
    \path [->] (L) edge node[below right, fill=none]{$\scriptsize\begin{matrix}
0 \\
n^2-2n+1 \\
n^2 \\
\end{matrix}$}(M);
    \path [->] (M) edge node[above, fill=none]{$\scriptsize\begin{matrix}
n^2 \\
0 \\
n^2-2n+1 \\
\end{matrix}$}(K);
   
\end{scope}

\begin{scope}[>={Stealth[red]}, %%THE SUBGRAPH OF THE TAIL OF TWO RADIX EXPANSIONS
              every node/.style={fill=white},
              every edge/.style={draw=red, thick}]

    \path [->] (C) edge node[right, fill=none]{$\scriptsize\begin{matrix}
0 \\
0 \\
n^2-2n+1 \\
\end{matrix}$+}(D);
 \path [->] (B) edge node[near end, right, yshift=0.5cm, fill=none]{$\scriptsize\begin{matrix}
0 \\
0 \\
2n-1 \\
\end{matrix}$+}(C);
\path [->] (B) edge node[above, fill=none]{$\scriptsize\begin{matrix}
0 \\
0 \\
2n \\
\end{matrix}$+}(G);
 \path [->] (E) edge node[near end, left, yshift=0.5cm, fill=none]{$\scriptsize\begin{matrix}
2n-1 \\
2n-1 \\
0 \\
\end{matrix}$+}(F);
\path [->] (C) edge[transform canvas={yshift=1mm}] node[above, fill=none]{$\scriptsize\begin{matrix}
0 \\
0 \\
n^{2}-2n+2 \\
\end{matrix}$+} (F);
   \path [->] (F) edge[transform canvas={yshift=-1mm}] node[below, fill=none]{$\scriptsize\begin{matrix}
n^2-2n+2 \\
n^2-2n+2 \\
0 \\
\end{matrix}$+} (C);
 \path [->] (E) edge node[above, fill=none]{$\scriptsize\begin{matrix}
2n \\
2n \\
0 \\
\end{matrix}$+}(D);
 \path [->] (F) edge node[left, fill=none]{$\scriptsize\begin{matrix}
n^2-2n+1 \\
n^2-2n+1 \\
0 \\
\end{matrix}$+}(G);  
   \path [->] (G) edge node[below right, fill=none]{$\scriptsize\begin{matrix}
0 \\
0 \\
n^2 \\
\end{matrix}$}(E);
 \path [->] (D) edge node[below left, fill=none]{$\scriptsize\begin{matrix}
n^2 \\
n^2 \\
0 \\
\end{matrix}$}(B); 
\end{scope}
\end{tikzpicture}
\caption{The graph governing equivalent radix expansions in base $-n+i$ for $n\geq3$.}
\label{fig:radixrepeat}
\end{figure}

%%REPEAT N=3 FIGURE
\begin{lemma}\label{lem:statesrepeat} [W. J. Gilbert, \cite{G82}, proposition 1]%%ALLOWABLE STATES

Let $n$ be a postive integer. Two radix expansions, $q$ and $r$, represent the same complex number in base $b = -n+i$ if and only if, for all integers $k$, either

\begin{itemize}\label{lem:statesrepeat}
\item[(i)] $q(k)-r(k) \in \{0, \pm1, \pm(n+i), \pm(n-1+i)\}$ when $n\neq 2$, or
\item[(ii)] $q(k)-r(k) \in \{0, \pm1, \pm(2+i), \pm(1+i), \pm i, \pm (2+2i)\}$ when $n = 2$. 
\end{itemize}
\end{lemma}

We proceed under the assumption that $n\geq3$. We discuss the special case of $n=2$ in Appendix~\ref{app:b}.  In \cite{G82}, Gilbert gives some of the calculations pertaining to the $n=1$ state graph. The derivation of that graph does not exhibit all the reasoning featured in the derivation of the graph governing the cases $n\geq3$. 

Let $p, q,$ and $r$ be radix expansions in base $b=-n+i$. The $k$th state is defined to be $S(k) := (p(k)-q(k), q(k)-r(k), r(k)-p(k))$. It is important to recall that, in this context, the index $k$ ranges over all the integers and the digit $p_{k}$ corresponds to the coefficient of $b^{k}$. 

Although the sum of the components of $S(k)$ is zero, our notation lists them all. This is because we wish to explicitly compute the digits of all three expansions in the $k$th place. We recall (\ref{eq:states}) from Section~\ref{section:complexCase}.
\begin{equation} \label{eq:graphkey}
S(k) = (p_{k}-q_{k}, q_{k}-r_{k}, r_{k}-p_{k}) + bS(k+1).
\end{equation}

It says that the $(k+1)$st state can be used to find the possible values of $S(k)$. Every radix expansion $d$ has a smallest index $\ell$ at which $d_{k} = 0$ for all $k\geq\ell$. Therefore there exists a $k$ for which $p(k+1) = q(k+1) = r(k+1) = 0$ and thus $S(k+1) = (0, 0, 0)$. This state corresponds with the top node of Figure \ref{fig:radixrepeat} with the diagram
$$\begin{tikzpicture}\draw (0,0) rectangle node{pqr} (.75,.75); \end{tikzpicture}.$$
We compute, using Lemma~\ref{lem:statesrepeat}, the possible values of $S(k)$. Each value will correspond to a node in the state graph that is a successor of the node corresponding to $S(k+1) = (0, 0, 0)$. 

Observe that by (\ref{eq:graphkey}) the $k$th state must satisfy $S(k) = (p_{k}-q_{k}, q_{k}-r_{k}, r_{k}-p_{k})$. This forces the components of $S(k)$ to be integers since each digit is an integer. In accordance with Lemma~\ref{lem:statesrepeat}, the components must be $0$ or $\pm1$. This splits into cases. It is that either all three digits are the same ($S(k) = (0, 0, 0)$) or at least one digit differs from the other two. 

The case $S(k) = (0, 0, 0)$ implies the existence of an arrow from the state $(0, 0, 0)$ back to itself. The triple of digits $(p_{k}, q_{k}, r_{k})$ could be any $(a, a, a)$ where $a\in\{0, 1, \ldots, n^{2}\}$. This is indicated by the label on the corresponding edge in the state graph given by 
$$\begin{matrix} 0\\0\\0\end{matrix}+.$$
We proceed with the case of differing digits. The digits cannot all be distinct because this would mean one of the pairs would necessarily have a difference of magnitude greater than or equal to $2$. Without loss of generality, let us say that $r$ is the expansion that differs in the $k$th digit and $p_{k}=q_{k}$. Either $r_{k}$ is one more than $p_{k}$ or one less. We either have $S(k) = (0, -1, 1)$ or $S(k) = (0, 1, -1)$. These states correspond to the diagrams 
$$\begin{tikzpicture}\draw (0,0) rectangle node{pq} (0.75,0.75); \draw (0.75, 0.75) rectangle  node{r} (1.5, 0); \end{tikzpicture}\;\;\text{and}\;\;\begin{tikzpicture}\draw (0,0) rectangle node{r} (0.75,0.75); \draw (0.75, 0.75) rectangle  node{pq} (1.5, 0); \end{tikzpicture}$$ respectively and result in the remaining two edges from the top node in Figure~\ref{fig:radixrepeat}. 

The triples $(p_{k}, q_{k}, r_{k})$ are either of the form $(a, a, a + 1)$ or $(a + 1, a + 1, a)$ where $a\in\{0, 1, \ldots, n^{2}-1\}$. This is indicated by the respective labels 
$$\begin{matrix} 0\\0\\1\end{matrix}+\;\;\text{and}\;\;\begin{matrix} 1\\1\\0\end{matrix}+$$
 on the corresponding edges.

This first step provides the flavour of the calculations that appear in the full derivation of the graph. We compute a second step which will include the possibility that all three of the digits $p_{k}, q_{k}$, and $r_{k}$ are distinct. Let us reindex such that $S(k+1) = (0, 1, -1)$. Again, we refer to (\ref{eq:graphkey}) to direct our calculations. We have
\begin{equation}\label{eq:potentialtriple}
S(k) = (p_{k}-q_{k}, q_{k}-r_{k}, r_{k}-p_{k}) + (0, -n+i, n-i).
\end{equation}
It is clear that, at least one of the digits must differ from the other two. Let us investigate the case of exactly one distinct digit. Without loss of generality we assume $p_{k} = q_{k}$ and $r_{k}\neq p_{k}$. Consider the second component of $S(k)$: $q_{k} - r_{k} - n + i$. 

The digits are integers and thus there is no way of changing the positive imaginary part. According to Lemma~\ref{lem:statesrepeat}, we can choose digits $q_{k}$ and $r_{k}$ such that $q_{k} - r_{k} = 2n$ or $2n-1$. The choice of a difference of $2n$ implies that the third component is $-n-i$, which satisfies Lemma~\ref{lem:statesrepeat}. The resulting state is $S(k) = (0, n+i, -n-i)$. Its corresonding diagram in Figure~\ref{fig:radixrepeat} is 
$$\begin{tikzpicture}\draw (0,0) rectangle node{r} (0.75,0.75); \draw (0.75, 0.75) rectangle  node{pq} (1.5, 1.5); \end{tikzpicture}.$$

The triple of digits $(p_{k}, q_{k}, r_{k})$ is of the form $(2n + a, 2n + a, a)$ where $a\in\{0, 1, \ldots, n^{2}-2n\}$. This is indicated by the label on the corresponding edge given by 
$$\begin{matrix} 2n\\2n\\0\end{matrix}+.$$

If we made the other choice, the resulting state is $S(k) = (0, n-1+i, -n+1-i)$ whose diagram is given by
$$\begin{tikzpicture}\draw (0,0) rectangle node{pq} (0.75,0.75); \draw (0, 0) rectangle  node{r} (.75, -.75); \end{tikzpicture}\;\;\text{and has the label}\;\;\begin{matrix} 2n-1\\2n-1\\0\end{matrix}+$$ on the incoming edge. 

Now we consider the case where all three are different and, in particular, $p_{k}\neq q_{k}$. We can see in (\ref{eq:potentialtriple}) that the first component of $S(k)$ is precisely $p_{k}-q_{k}$. It follows from Lemma~\ref{lem:statesrepeat} that either $p_{k}$ is one more than $q_{k}$ or one less. The expansions $p$ and $q$ have the same digits for all places $k+j$ for all $j\geq1$. We are distinguishing them for the first time. Without loss of generality we may assume $p_{k} = q_{k} - 1$. 

In order for the remaining components of $S(k)$ to obey Lemma~\ref{lem:statesrepeat}, we must have $q_{k} - r_{k} = 2n$ and thus $r_{k} - p_{k} = -2n+1$. The resulting state is $S(k) = (1, n-1+i, -n-i)$ and its corresponding diagram is $$\begin{tikzpicture}\draw (0,0) rectangle node{r} (0.75,0.75);\draw (0.75, 0.75) rectangle node{p} (0, 1.5); \draw (0.75, 0.75) rectangle node{q} (1.5, 1.5); \end{tikzpicture}.$$

The remaining structure of the state graph can be deduced by iterating this procedure until all the successive states are found. We leave this task to the interested reader. 
\newpage
\section{The Other State Graph ($n=2$)}\label{app:b}%%APPENDIX B

This appendix is a supplement to the discussion of Figure~\ref{fig:radix} in Section~\ref{section:complexCase}. Here we present the state graph governing equivalent radix expansions in base $-2+i$. 

In Lemma~\ref{lem:lemstates}, the difference $p(k)-q(k)$ may take on a larger number of values when $n=2$. This increases the number of realizable states and thus complicates the corresponding state graph. The method used to derive the state graph for $n\geq3$ applies in the case $n=2$. We do not include the details. We do include the notation required to parse the diagrams for the new states in the state graph, the primary claim from \cite{G82} about the graph (Theorem~\ref{thm:rules2}), and the graph itself (Figure~\ref{fig:radix2}). The new edges particular to $n=2$ are highlighted in blue and any successor of a blue edge is also a new state particular to the $n=2$ case. 

We make special mention that we only label the edges that correspond to the first distinction between a pair of expansions. The interested reader can 
derive any edge label using the value of the source and successor states of the edge and (\ref{eq:graphkey}). 

Let $p$ and $q$ be two radix expansions in base $-2+i$. We extend the list of diagrams from Section~\ref{section:complexCase} that communicate the value of $p(k)-q(k)$. The additions are as follows:
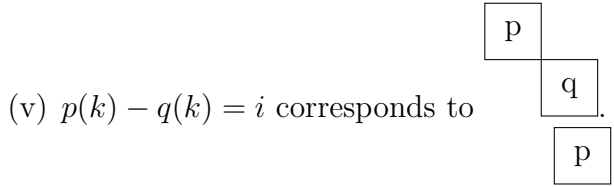
\begin{enumerate}
\item[(v)]  $p(k)-q(k) = i$ corresponds to \begin{tikzpicture}\draw (0,0) rectangle node{q} (.75,.75); \draw (0, .75) rectangle node{p} (-0.75, 1.5); \end{tikzpicture}.
\item[(vi)] $p(k)-q(k) = 2+2i$ corresponds to  
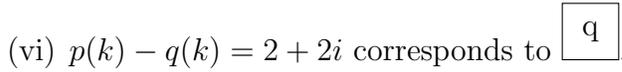
\begin{tikzpicture}\draw (0,0) rectangle node{q} (0.75,0.75); \draw (0, 1.5) rectangle  node{p} (0.75, 2.25); \end{tikzpicture}. %%INCLUDE STATE EXAMPLE AS IN THE BODY
\end{enumerate}
We can communicate the value of additional states using these diagrams. For example, the state $(-1-i, 1+i, -2-2i)$ is communicated by the diagram
$$\begin{tikzpicture}\draw (0,0) rectangle node{r} (0.75,0.75);\draw (0.75,0.75) rectangle node{q} (0,1.5);\draw (0,1.5) rectangle node{p} (0.75,2.25);\end{tikzpicture}.$$

\begin{theorem}\label{thm:rules2}[W. J. Gilbert, \cite{G82}, theorem 8]%%RADIX EXPANSIONS RULES
Let $p, q$ and $r$ be three radix expansions in base $-2+i$. These expansions represent the same complex number if and only if they can be obtained from an infinite path through the state graph in Figure~\ref{fig:radix2} starting at state $(0, 0, 0)$, if necessary relabelling $p, q$ and $r$ and in some cases, when $p=q$, replacing $q$ with another expansion. 
\end{theorem}

\begin{figure} [h]%%GRAPH GOVERNING RADIX EXPANSIONS n=2
\centering
\vspace*{-3em}
\hspace*{-5.5em}
\begin{tikzpicture}
%\useasboundingbox (-11, -15) rectangle (11,7);
\begin{scope}[every node/.style]
    \node (A) at (0, 9) {\begin{tikzpicture}\draw (0,0) rectangle node{pqr} (.75,.75); \end{tikzpicture}}; %TOP NODE
    \node (B) at (-4.5,6) {\begin{tikzpicture}\draw (0,0) rectangle node{pq} (0.75,0.75); \draw (0.75, 0.75) rectangle  node{r} (1.5, 0); \end{tikzpicture}}; %LEFT OF A
    \node (C) at (-4.5,0) {\begin{tikzpicture}\draw (0,0) rectangle node{pq} (0.75,0.75); \draw (0.75, 0.75) rectangle node{r} (0, 1.5); \end{tikzpicture}}; %DOWN FROM B
    \node (D) at (-2,2.75) {\begin{tikzpicture}\draw (0,0) rectangle node{r} (0.75,0.75); \draw (0.75, 0.75) rectangle node{pq} (1.5, 1.5); \end{tikzpicture}}; %UP-RIGHT FROM C
    \node (E) at (4.5,6) {\begin{tikzpicture}\draw (0,0) rectangle node{r} (0.75,0.75); \draw (0.75, 0.75) rectangle node{pq} (1.5, 0); \end{tikzpicture}}; %RIGHT OF A
    \node (F) at (4.5,0) {\begin{tikzpicture}\draw (0,0) rectangle node{r} (0.75,0.75); \draw (0.75, 0.75) rectangle node{pq} (0, 1.5); \end{tikzpicture}}; %DOWN FROM E
    \node (G) at (2, 2.75) {\begin{tikzpicture}\draw (0,0) rectangle node{pq} (0.75,0.75); \draw (0.75, 0.75) rectangle node{r} (1.5, 1.5); \end{tikzpicture}}; %UP-LEFT FROM F

    \node (H) at (-6,-5.5) {\begin{tikzpicture}\draw (0,0) rectangle node{p} (0.75,0.75); \draw (0.75, 0.75) rectangle node{r}(0, 1.5);\draw (0, 0) rectangle node{q} (-0.75, 0.75); \end{tikzpicture}}; %TOP OF TRIPLE (LHS)
    \node (I) at (-6, -9.25) {\begin{tikzpicture}\draw (0,0) rectangle node{q} (0.75,0.75); \draw (0.75, 0.75) rectangle  node{p} (0, 1.5);\draw (0, 0) rectangle node{r} (-0.75, 0.75); \end{tikzpicture}}; %DOWN FROM H
    \node (J) at (-8.25,-7.75) {\begin{tikzpicture}\draw (0,0) rectangle node{r} (0.75,0.75); \draw (0.75, 0.75) rectangle node{q} (0, 1.5); \draw (0, 0) rectangle node{p} (-0.75, 0.75); \end{tikzpicture}}; %UP-LEFT FROM I
    \node (K) at (6, -5.5) {\begin{tikzpicture}\draw (0,0) rectangle node{r} (0.75,0.75);\draw (0.75, 0.75) rectangle node{p} (0, 1.5); \draw (0.75, 0.75) rectangle node{q} (1.5, 1.5); \end{tikzpicture}}; %TOP OF TRIPLE (RHS)
    \node (L) at (6, -9.25) {\begin{tikzpicture}\draw (0,0) rectangle node{p} (0.75,0.75); \draw (0.75, 0.75) rectangle node{q}(0, 1.5);  \draw (0.75, 0.75) rectangle node{r} (1.5, 1.5); \end{tikzpicture}}; %DOWN FROM K
    \node (M) at (8.25, -7.75) {\begin{tikzpicture}\draw (0,0) rectangle node{q} (0.75,0.75); \draw (0.75, 0.75) rectangle node{r} (0, 1.5);  \draw (0.75, 0.75) rectangle node{p} (1.5, 1.5); \end{tikzpicture}}; %UP-RIGHT FROM L

    \node (N) at (-9.5, -5) {\begin{tikzpicture}\draw (0,0) rectangle node{q} (0.75,0.75);\draw (0.75,0.75) rectangle node{r} (0,1.5);\draw (0.75,0.75) rectangle node{p} (1.5,0); \end{tikzpicture}}; %MOST LEFT DISTINCT TRIPLE
    \node (O) at (-9.5, -11) {\begin{tikzpicture}\draw (0,0) rectangle node{r} (0.75,0.75);\draw (0.75,0.75) rectangle node{q} (0,1.5);\draw (0,1.5) rectangle node{p} (0.75,2.25);\end{tikzpicture}}; %DOWN FROM N
    \node (P) at (9.5, -5)  {\begin{tikzpicture}\draw (0,0) rectangle node{p} (0.75,0.75);\draw (0.75,0.75) rectangle node{q} (1.5,0);\draw (0.75,0) rectangle node{r} (1.5,-0.75);\end{tikzpicture}}; %RIGHT MOST DISTINCT TRIPLE
   \node (Q) at (9.5, -11) {\begin{tikzpicture}\draw (0,0) rectangle node{p} (0.75,0.75);\draw (0.75,0.75) rectangle node{q} (0,1.5);\draw (0,1.5) rectangle node{r} (0.75,2.25);\end{tikzpicture}}; %DOWN FROM P

    \node (R) at (-3, -5.5) {\begin{tikzpicture}\draw (0,0) rectangle node{r} (0.75,0.75);\draw (0.75,0.75) rectangle node{p} (0,1.5);\draw (0.75,0.75) rectangle node{q} (1.5,0); \end{tikzpicture}}; %RIGHT OF H
    \node (S) at (-1, -5.5) {\begin{tikzpicture}\draw (0,0) rectangle node{p} (0.75,0.75);\draw (0.75,0.75) rectangle node{r} (0,1.5);\draw (0,1.5) rectangle node{q} (0.75,2.25);\end{tikzpicture}}; %RIGHT OF R
    \node (T) at (-3, -9.25) {\begin{tikzpicture}\draw (0,0) rectangle node{p} (0.75,0.75);\draw (0.75,0.75) rectangle node{q} (0,1.5);\draw (0.75,0.75) rectangle node{r} (1.5,0); \end{tikzpicture}}; %RIGHT OF I
    \node (U) at (-1, -9.25) {\begin{tikzpicture}\draw (0,0) rectangle node{q} (0.75,0.75);\draw (0.75,0.75) rectangle node{p} (0,1.5);\draw (0,1.5) rectangle node{r} (0.75,2.25);\end{tikzpicture}}; %RIGHT OF T

    \node (V) at (3, -5.5) {\begin{tikzpicture}\draw (0,0) rectangle node{q} (0.75,0.75);\draw (0.75,0.75) rectangle node{r} (1.5,0);\draw (0.75,0) rectangle node{p} (1.5,-0.75);\end{tikzpicture}}; %LEFT OF K
    \node (W) at (1, -5.5) {\begin{tikzpicture}\draw (0,0) rectangle node{q} (0.75,0.75);\draw (0.75,0.75) rectangle node{r} (0,1.5);\draw (0,1.5) rectangle node{p} (0.75,2.25);\end{tikzpicture}}; %LEFT OF V
    \node (X) at (3, -9.25) {\begin{tikzpicture}\draw (0,0) rectangle node{r} (0.75,0.75);\draw (0.75,0.75) rectangle node{p} (1.5,0);\draw (0.75,0) rectangle node{q} (1.5,-0.75);\end{tikzpicture}}; %LEFT OF L
    \node (Y) at (1, -9.25) {\begin{tikzpicture}\draw (0,0) rectangle node{r} (0.75,0.75);\draw (0.75,0.75) rectangle node{p} (0,1.5);\draw (0,1.5) rectangle node{q} (0.75,2.25);\end{tikzpicture}}; %LEFT OF X
   
\end{scope}

\begin{scope}[>={Stealth[black]},
              every node/.style={fill=white},
              every edge/.style={draw=black}]
    \path [->] (A) edge[loop above]  node[fill=none]{} (A);
    \path [->] (A) edge node[above, fill=none]{$\scriptsize\begin{matrix}
0 \\
0 \\
1 \\
\end{matrix}$+}(B);
    \path [->] (A) edge node[above, fill=none]{$\scriptsize\begin{matrix}
1 \\
1 \\
0 \\
\end{matrix}$+}(E);
    \path [->] (B) edge node[left, fill=none]{$\scriptsize\begin{matrix}
1 \\
0 \\
4 \\
\end{matrix}$}(H);
    \path [->] (E) edge node[right, fill=none]{$\scriptsize\begin{matrix}
3 \\
4 \\
0 \\
\end{matrix}$}(K);
    \path [->] (C) edge node[below, yshift=-0.25cm, xshift=3.5mm, fill=none]{$\scriptsize\begin{matrix}
0 \\
1 \\
2 \\
\end{matrix}$+}(K); 
    \path [->] (F) edge node[below, yshift=-0.25cm, xshift=-3mm, fill=none]{$\scriptsize\begin{matrix}
2 \\
1 \\
0 \\
\end{matrix}$+}(H);

    \path [->] (H) edge node[fill=none]{}(I);
    \path [->] (I) edge node[fill=none]{}(J);
    \path [->] (J) edge node[fill=none]{}(H);
    \path [->] (K) edge node[fill=none]{}(L);
    \path [->] (L) edge node[fill=none]{}(M);
    \path [->] (M) edge node[fill=none]{}(K);
   
\end{scope}

\begin{scope}[>={Stealth[black]}, %%THE SUBGRAPH OF THE TAIL OF TWO RADIX EXPANSIONS
              every node/.style={fill=white},
              every edge/.style={draw=black}]

    \path [->] (C) edge node[fill=none]{}(D);
 \path [->] (B) edge node[fill=none]{}(C);
\path [->] (B) edge node[fill=none]{}(G);
 \path [->] (E) edge node[fill=none]{}(F);
\path [->] (C) edge[transform canvas={yshift=1mm}] node[fill=none]{} (F);
   \path [->] (F) edge[transform canvas={yshift=-1mm}] node[fill=none]{} (C);
 \path [->] (E) edge node[fill=none]{}(D);
 \path [->] (F) edge node[fill=none]{}(G);  
   \path [->] (G) edge node[fill=none]{}(E);
 \path [->] (D) edge node[fill=none]{}(B); 
\end{scope}

\begin{scope}[>={Stealth[blue]}, %%THE ADDITIONAL STRUCTURE OF THE N=2 CASE
              every node/.style={fill=white},
              every edge/.style={draw=blue, thick}]

    \path [->] (B) edge node[left, fill=none]{$\scriptsize\begin{matrix}
1 \\
0 \\
3 \\
\end{matrix}$+}(N);
 \path [->] (N) edge node[fill=none]{}(O);
 \path [->] (O) edge[transform canvas={yshift=1mm}] node[fill=none]{}(Q);
 \path [->] (E) edge node[right, xshift=1.5mm, fill=none]{$\scriptsize\begin{matrix}
2 \\
3 \\
0 \\
\end{matrix}$+}(P);
 \path [->] (P) edge node[fill=none]{}(Q);
 \path [->] (Q) edge [transform canvas={yshift=-1mm}] node[fill=none]{}(O);
 \path [->] (J) edge node[fill=none]{}(N);
 \path [->] (M) edge node[fill=none]{}(P);
 \path [->] (H) edge node[fill=none]{}(R);
 \path [->] (R) edge node[fill=none]{}(S);
 \path [->] (S) edge[transform canvas={yshift=1mm}] node[fill=none]{}(W);
 \path [->] (K) edge node[fill=none]{}(V);
 \path [->] (V) edge node[fill=none]{}(W);
 \path [->] (W) edge[transform canvas={yshift=-1mm}] node[fill=none]{}(S);
 \path [->] (I) edge node[fill=none]{}(T);
 \path [->] (T) edge node[fill=none]{}(U);
 \path [->] (U) edge[transform canvas={yshift=1mm}] node[fill=none]{}(Y);
 \path [->] (L) edge node[fill=none]{}(X);
 \path [->] (X) edge node[fill=none]{}(Y);
 \path [->] (Y) edge[transform canvas={yshift=-1mm}] node[fill=none]{}(U);

 \path [->] (F) edge node[near end, above, yshift = 5mm, xshift=20mm, fill=none]{$\scriptsize\begin{matrix}
3 \\
2 \\
0 \\
\end{matrix}$+}(N);
 \path [->] (C) edge node[near end, above , yshift=5.75mm, xshift=-20mm, fill=none]{$\scriptsize\begin{matrix}
0 \\
1 \\
3 \\
\end{matrix}$+}(P);
\end{scope}
\end{tikzpicture}
\vspace{-15mm}
\caption{The graph governing equivalent radix expansions in base $-2+i$.}
\label{fig:radix2}
\end{figure}

\end{document}